\documentclass[12pt]{amsart}

\usepackage{epsfig}
\usepackage{fullpage}
\usepackage{latexsym}

\usepackage{caption}
\usepackage{url}
\usepackage{color}
\usepackage{amssymb}
\usepackage{amsmath}
\usepackage{amsthm}
\usepackage{graphicx}
\usepackage{bbm}

\newtheorem{theorem}{Theorem}[section]
\newtheorem{lemma}[theorem]{Lemma}

\newtheorem{defn}[theorem]{Definition}
\newtheorem{definition}[theorem]{Definition}

\newtheorem{corollary}[theorem]{Corollary}
\newtheorem{claim}[theorem]{Claim}
\newtheorem{remark}[theorem]{Remark}
\newtheorem{conjecture}[theorem]{Conjecture}
\newtheorem{proposition}[theorem]{Proposition}
\newtheorem{counterex}[theorem]{Counterexample}

\allowdisplaybreaks

\DeclareMathOperator{\vd}{vdiam}

\DeclareMathOperator{\mult}{mult}
\DeclareMathOperator{\add}{add}
\DeclareMathOperator{\tl}{tl}
\def\qed{\hfill
\ifhmode\unskip\nobreak\fi\quad\ifmmode\Box\else$\Box$\fi\\ }

\begin{document}
\title{Negatively Curved Graphs}
\author{Matthew P. Yancey $^*$}

\maketitle

\begin{footnotesize}
\noindent $^*$ Institute for Defense Analyses - Center for Computing Sciences, Bowie MD, USA; mpyancey1@gmail.com

\end{footnotesize}

\begin{abstract}
Many applications in network science have recently been discovered for the ``curvature'' of a network, but there is no consensus on the definition for this term.
A common approach in these applications is to derive from the curvature either a ``logical center'' of the network or a tree representation of the network (these would only exist when the curvature is negative), but that such structures can be extracted using curvature alone remains largely conjectural.
A connection between one type of curvature---Gromov's hyperbolicity---and a tree representation has been known for decades, and recently it has also been connected for unweighted graphs to a logical center.
We extend the connection between Gromov's hyperbolicity and a logical center to weighted graphs, and we construct counterexamples showing that no other proposed definition for curvature implies the existence of a logical center.
We also consider the leading methods to construct a tree representation of the network and the leading methods to measure the quality of the representation, and show that, despite wildly different descriptions, they are asymptotically equivalent.

These results resolve several conjectures, including a conjecture of Dourisboure and Gavoille on a $2$-approximation method for calculating tree-length and all of the conjectures from Jonckheere, Lou, Bonahon, and Baryshnikov relating congestion to rotational symmetry.
\end{abstract}

AMS Classification: 51K99, 05C99

\section{Motivation}
There has been an attempt recently to describe routing in computer networks by modeling the traffic as a set of shortest paths in non-Euclidean space.
To see how this comparison has deep qualitative significance, consider the model using the Poincar\'{e} disk.
The metric on the Poincar\'{e} Disk is smaller (``cheaper'') between points closer to the origin, and so all shortest paths arc towards the center as they travel.
In computer networks, bits of information that need to travel far will be passed from regional hubs up to backbone routers, which is a subnetwork designed to route a large volume of traffic cheaply across long distances.
Each successive stage up is then considered closer to the ``logical center'' of the network.
But in what tangible ways is this metaphor useful?

Shavitt and Tankel \cite{ST}, and later Bogu\~{n}\'{a}, Papadopoulos, and Krioukov \cite{BPK}, described a heuristic method to map nodes of a computer network into a $2$-dimensional hyperbolic space.
They started with placing vertices with highest degree (a parameter known to correlate with being in the backbone) at the center of the hyperbolic space, and then moved on to lower degree vertices, placing them at larger radii from the center.
The goal of the project was a mapping of telecommunication companies to a hyperbolic space that could be used to efficiently model Internet traffic.
The experiment was a success.
Another example of using a curved space to solve a problem is the scheme by Gavoille and Ly \cite{GL} for distance-labeling networks.

After the appearance of these applications various types of curved spaces were presented as alternative methods to model big data or big networks.
Most of these spaces have a parameter, usually called the curvature of the space.
Proposed uses of curvature include community detection \cite{EM}, network security \cite{J,JL}, quantum information \cite{JLS,JSL}, influence in regulatory networks \cite{ADM,BCC}, graph mining \cite{AD}, connectivity of sensor networks \cite{ALJKZ}, and network modeling \cite{KPKVB,CF}.
These projects use discretizations of negative curvature; we investigate discretizations of positive curvature in a companion paper \cite{Y_pos}.

We model all networks as graphs with a vertex set $V$ and edges $E$ with edge weights $w:E \rightarrow \mathbb{R}_{\geq 0}$.
The length of a path is the sum of the weights on its edges,\footnote{Weights on edges have been used to symbolize many different types of quantities.  Here we use it to represent length.  This is the opposite of using weights to represent similarity, which is also common.} and the distance between two points is the length of the shortest path between them.
Edge weighted graphs are sufficiently general to model all finite metric spaces.
A metric space is a \emph{tree metric} if it can be represented using graph distance where the underlying graph is a tree.
Tree metrics are also known as additive metrics.

Most of the above projects use curvature to either find a ``center'' of the network or establish a tree metric over the nodes of the network that approximates the standard graph distance.
One motivation for locating the center is that it provides a manner in which the algorithms of Shavitt and Tankel \cite{ST} and Bogu\~{n}\'{a}, Papadopoulos, and Krioukov \cite{BPK} to embed a network in hyperbolic space can be used on an arbitrary hyperbolic network without collateral information about a backbone.
This motivated links to be conjectured to exist between a well-defined ``center''  and each of hyperbolic spaces \cite{JLA,JLB}, curved $2$-manifolds \cite{NS}, and fixed points in the automorphism group of the graph \cite{JLBB}.
Finding a tree metric that approximates the structure of a given network is a topic that predates graph curvature by several decades and has its own applications to routing schemes \cite{CDEHVX,D}, evolutionary trees \cite{ABFPT,DHHKMW}, graph mining \cite{AD}, computational geometry \cite{BIS}, and graph visualization \cite{DG}.

A key thing the reader should notice in the above literature review is the breadth of ideas rather than the depth.
That is, there are many different approaches to tree metrics or analogues for curvature mentioned above, but none of these ideas have been refined through repeated study. 
The motivation of this paper is to simplify the disparate apply-curvature-to-discrete-spaces field through proving, strengthening, or disproving connections among some of the known algorithms, parameters, and results.
This is by no means a survey, and for brevity we only touch on ideas for which we have results.
Proving or disproving more connections between other parameters and results in the field would be interesting.

The results presented here are of two forms: (A) we unify into a single theory the different definitions for curvature that are rigorously connected to congestion, several of the different measurements for the quality of a tree metric approximation, and all of the corresponding best-known algorithms for constructing an approximate tree metric, and (B) we generalize the result of Chepoi, Dragan, and Vax\`{e}s \cite{CDV} about the existence of a ``center'' in a hyperbolic network from unweighted graphs (i.e. $w \equiv 1$) to general metric spaces (the changes also lead to a faster algorithm to find the center).

\subsection{Congestion}

We model traffic on a computer network using \emph{uniformly distributed demand}.
We assume all graphs are connected.
Let $P(u, v)$ denote the number of shortest paths from $u$ to $v$, and $P(u,v;S)$ denote the number of shortest paths from $u$ to $v$ that intersect the vertex set $S$.
The demand on $S$ is defined to be $\mathbb{D}(S) = \sum_{u,v} P(u,v;S)/P(u,v)$.
Demand is related to the popular parameter \emph{betweenness centrality} of vertex $w$, which can be calculated as $\mathbb{D}(\{w\})/{|V| \choose 2}$.
As $P(u,v,\{v\}) = P(u,v)$ for each $u$ we have that $\mathbb{D}(\{v\}) \geq |V|-1$ for each $v$.
In most real world networks the average distance between two vertices is small \cite{LKF}, and therefore the average amount of demand on an individual vertex is approximately $c|V|$ for some constant $c$.
A \emph{congested} network is then a network with a skewed distribution of demand, and intuitively a ``center'' of the network (if it exists) would be a small connected set $S$ with $\mathbb{D}(S)$ is $\Theta(|V|^2)$.

We use two formal definitions for the ``center.''
The first is used to identify a center in some fixed network, and the second is used to evaluate whether centers exist (i.e. is the network congested) in an asymptotic sense in an infinite family of networks.
The first definition comes from Chepoi, Dragan, and Vax\`{e}s \cite{CDV}.
The second definition is our own\footnote{The intuition and early definitions for curvature and congestion are based on infinite graphs, which become trivial when directly applied to finite graphs.  The motivation for new definitions stems from the applications to finite networks.  The definition we invented is inspired by the concept of ``finite yet unbounded.''}, and our best interpretation of the intuition generated in \cite{JLA,JLB,NS,JLBB}.

Let $d(x,y)$ denote the distance between vertices $x,y$.
For a vertex subset $S$ and a vertex $x$, we use the notation $d(x,S) = \min\{d(y,x):y \in S\}$.
The closed ball of radius $r$ around vertex subset $S \subseteq V$ is denoted $B_r(S) = \{y \in V: d(y,S) \leq r\}$.
Because we work with finite metric spaces, it is more natural to work with closed balls, and thus a \emph{ball} will refer to a closed ball.

\begin{definition}[Chepoi, Dragan, and Vax\`{e}s \cite{CDV}]
Fix a graph $G$ and a vertex subset $X \subseteq V$.
An $(\alpha,r)-$\emph{core} is a ball around a vertex $m$ of radius $r$ such that for the set $M = \{x,y \in X: P(x,y;B_r(m)) = P(x,y)\}$ we have $|M| \geq \alpha {|X| \choose 2}$.
We say that $m$ is the center of such a core.
\end{definition}

As a slight abuse of notation, we may refer to $m$ as the core instead of $B_r(m)$.
An $(\alpha,r)$-core may not exist, and it may not be unique if it does exist.
For simplicity, we assume through out this paper that $X=V$.

\begin{definition}
A family of graphs $\{G_i\}$ is \emph{congested} if there exists an $\epsilon > 0$ such that for each $i$, there exists a vertex $v_i \in V(G_i)$ such that $\mathbb{D}(\{v_i\}) > \epsilon |V(G_i)|^2$.
We say that $\{G_i\}$ is \emph{roughly congested} if there exists fixed $\epsilon > 0$ and $r$ such that for each $i$, there exists a vertex $v_i \in V(G_i)$ such that $\mathbb{D}(B_r(v_i)) > \epsilon |V(G_i)|^2$.
\end{definition}

If a family of graphs is roughly congested, then we describe $v_i$ as a center of the core in $G_i$.
Centers of cores may not be unique.

Gromov's $4$-point hyperbolicity of a graph $G$, denoted by $\delta(G)$, is the smallest $k$ such that  for any $4$ vertices $x,y,z,r$ we have that 
\begin{equation}\label{4 point condition}
d(x,z) + d(y,r) \leq \max\{d(x,y) + d(r,z), d(y,z) + d(x,r)\} + 2k.
\end{equation}
When there is no chance for confusion, we use $\delta$ to denote $\delta(G)$.
A central result in this field is that $\delta$-hyperbolicity implies the existence of cores.

\begin{theorem}[\cite{CDV}, Theorem 4.1 and Remark 3]\label{SODA full thm}
If $G$ is an unweighted $\delta$-hyperbolic graph, then there exists a $(\frac{1}{2}, 4 \delta)$-core.  
Moreover, finding the center of this core can be done in $O(|V|^2 |E|)$ time.
\end{theorem}

We unify the different definitions for curvature that are rigorously connected to congestion via a set of negative results.
That is, for each other definition of negative graph curvature---except for Gromov's hyperbolicity---we construct examples showing that bounding that form of curvature does not induce a family of graphs that is roughly congested.

Another result we prove is a variation on Theorem \ref{SODA full thm} for weighted graphs (as our result does not reduce to theirs when all edges are given weight $1$, the results are independent).
We use the notation $W = \max_{uv \in E}w(uv)$, and define $W_{/2} = \lfloor W/2 \rfloor$ when $w$ is integral-valued and $W/2$ otherwise.
By definition, $W_{/2} = 0$ if and only if $G$ is an unweighted graph.
Any generalization of Theorem \ref{SODA full thm} to weighted graphs must incorporate some function of $W$.
To see why, observe that a complete graph where every edge has weight $W \gg 1$ is $0$-hyperbolic but does not have an $(\alpha, r)$-core for $r < W$.

\begin{theorem}\label{full congestion}
Fix $0 < \alpha < 1/2$, and set $r(\alpha) = 2\delta(G) + 3W/2 + \frac{1-\alpha}{1-2\alpha}(4 \delta(G) + 2 W_{/2})$.
Every finite weighted graph $G$ contains an $(\alpha,r(\alpha))$-core.
Moreover, the center of such a core can be found in $O(|V||E|)$ time.
\end{theorem}

\subsection{Measuring Tree-likeness}

A beautiful aspect of mathematics is that particularly clever techniques have a way of showing up independently in multiple contexts, with those appearances only to be connected at some later date.
We consider three parameters that gauge how ``tree-like'' a graph $G$ is; they are \emph{additive distortion} $\gamma_{\add}(G)$, \emph{multiplicative distortion} $\gamma_{\mult}(G)$, and \emph{tree length} $\tl(G)$.
We will discuss algorithms that approximate these parameters, as it is NP-complete to compute tree-length \cite{L} or multiplicative distortion \cite{DY} and finding the argmin of additive distortion \cite{ABFPT} is NP-hard.

Two of the measures are based on \emph{distance approximating trees}.
Fix some graph $G$ and tree $\mathcal{T}(G)$.
A distance approximating tree is a map $T: V(G) \rightarrow V(\mathcal{T}(G))$ such that $|d_G(u,v) - d_{\mathcal{T}(G)}(T(u), T(v))|$ is small for all pairs of vertices $u,v \in V(G)$.
As an abuse of notation, we frequently will use $T$ to refer to both the map and the underlying tree of a distance approximating tree.
The additive distortion is defined by
$$ \gamma_{\add}(G) = \min_T \|G-T\|_\infty := \min_T \max_{u,v} \left|d_G(u,v) - d_T(T(u),T(v))\right|, $$
and the multiplicative distortion is defined by 
$$ \gamma_{\mult}(G) = \min_T \left\{ \max_{u,v}  \left\{ \frac{d_T(T(u),T(v)) }{d_G(u,v)} \right\} : d_T(T(x),T(y)) \geq d_G(x,y)\ \forall x,y\right\} .$$
The condition in the definition of multiplicative distortion that the tree be non-contracting is common (e.g., \cite{CDNRV}) but not universal (e.g., \cite{NR}), but it is folklore that by multiplying the weight of all edges by $\gamma_{\mult}^{-1/2}$ the two definitions become equivalent.

The third measure is based on tree decompositions.
For a tree $T$, let $S(T)$ be the set of all vertex subsets that correspond to a sub-tree of $T$.
A \emph{tree-decomposition} of a graph $G$ is a tree $T'$ and a map $\phi: V(G) \rightarrow S(T')$, such that if $xy \in E(G)$ then $\phi(x) \cap \phi(y) \neq \emptyset$.
The tree-length of a given decomposition is $\tl(T',\phi) = \max \{d(u,v) : \phi(u) \cap \phi(v) \neq \emptyset\}$, and $\tl(G) = \min_{(T, \phi)}\tl(T, \phi)$.
It has been established that these three measures are roughly equivalent for unweighted graphs.

\begin{theorem}[various sources---see Section \ref{History of tree-like algorithms subsec}]\label{all the same for unweighted}
There exists a universal constant $C$ such that if $G$ is an unweighted graph, then 
$$\max\{ \tl(G), \gamma_{\mult}(G), \gamma_{\add}(G)\} \leq C(1 + \min\{\tl(G), \gamma_{\mult}(G), \gamma_{\add}(G) \}).$$
\end{theorem}

The clever technique that has shown up independently in multiple contexts is not how to measure the ``tree-likeness'' of a graph, but how to construct a tree that resembles a given graph.
That is, Theorem \ref{all the same for unweighted} was established by discovering that a single algorithm to construct a distance approximating tree had been re-invented in the contexts of a constant factor approximation to each of $\tl(G), \gamma_{\mult}(G), \gamma_{\add}(G)$.

Our next result is an analogue of Theorem \ref{all the same for unweighted} for weighted graphs.
But first let us observe that a direct generalization is not possible, as the following standard modifications of edge weights affect the three parameters differently.
For any $uv \notin E$, adding edge $uv$ with weight $w(uv) > d(u,v)$ will not affect $\gamma_{\mult}(G)$ or $\gamma_{\add}(G)$ but it might affect $\tl(G)$.
Subdividing an edge might increase and will not decrease $\gamma_{\mult}(G)$ or $\gamma_{\add}(G)$ while it will not increase $\tl(G)$.\footnote{To see why $\tl(G)$ will not increase, suppose we subdivide edge $xy$, which creates vertex $z$.  By definition, an optimal tree decomposition of the original graph has $a \in \phi(x) \cap \phi(y)$.  Append leaf node $b$ to $a$ such that $\phi^{-1}(b) = \{x,y,z\}$.}
If we multiply every edge weight by a common factor, then $\tl(G)$ and $\gamma_{\add}(G)$ scale proportionally while $\gamma_{\mult}(G)$ is invariant.

While it takes a different form, Theorem \ref{connecting the parameters thm} does reduce to Theorem \ref{all the same for unweighted} on unweighted graphs.
Most surprisingly, it is proven in the same way: we establish a common framework between four of the current-best algorithms for constructing a tree-like representation of a graph.

We say that an edge $uv$ is \emph{relevant} if $w(uv) = d(u,v)$, otherwise it is \emph{irrelevant}.
In an unweighted graph, every edge is relevant.
If an edge is irrelevant, then deleting it will not affect the underlying metric.
The set of irrelevant edges can be determined in $O(|V|^3)$ time using the Floyd-Warshall algorithm.

\begin{theorem}\label{connecting the parameters thm}
There exists a universal constant $C$ such that any graph $G$ where every edge is relevant satisfies
$$\max\{ \tl(G), \gamma_{\mult}(G), \gamma_{\add}(G)\} \leq C( W + \min\{\tl(G), (1+W) \gamma_{\mult}(G), \gamma_{\add}(G) \}).$$
\end{theorem}

\subsection{Structure and Definitions}

In Section \ref{conjestion sec} we will prove Theorem \ref{full congestion}.
In Section \ref{distance approximating tree sec} we will prove Theorem \ref{connecting the parameters thm} by defining a new algorithm, and showing how it relates to the current-best approximation algorithms for each of additive distortion, multiplicative distortion, and tree-length.
As part of considering the current-best algorithms, we also resolve a conjecture about which algorithm to approximate tree-length has the best worst-case performance.
In Section \ref{curve sec} we construct counterexamples demonstrating that the existing definitions of curvature besides hyperbolicity are not sufficient to imply rough congestion, or in one case that the conjecture is true even without the assumption on curvature.

For vertices $u,v$, let $P_{u,v}$ denote an arbitrary (but fixed) shortest path between $u$ and $v$ in $G$.
We will use Gromov's product, which for vertices $a,b,r$ is 
$$(a,b)_r := \frac{d(r,a) + d(r,b) - d(a,b)}{2}.$$
As intuition for this formula: if the distance approximating tree $T$ has a subtree whose leaves are $\{T(a),T(b),T(r)\}$ and a (necessarily unique) vertex $u$ with degree $3$, then the value $(a,b)_r$ should be approximately the distance between $T(r)$ and $u$ (up to the quality of the tree).
Surprisingly, this intuition will actually be more useful while we prove congestion than when we study distance approximating trees.
Simple arithmetic shows that (\ref{4 point condition}) is equivalent to 
\begin{equation}\tag{1 alternate}
(x,z)_r \geq \min\{(x,y)_r, (y,z)_r \} - k.
\end{equation}

Let $n = |V|$.
We use standard notation for the neighborhood of a vertex, $N(u) = \{w : uw \in E\}$.
For (partial) tree-decomposition $(T, \phi)$, we define $\phi^{-1}:V(T) \rightarrow 2^{V(G)}$ to be $\phi^{-1}(v) = \{x \in V(G): v \in \phi(x)\}$.
If $U \subset V(T)$, then we define $\phi^{-1}(U) = \cup_{v \in U} \phi^{-1}(v)$.

\section{Cores in Weighted Graphs}\label{conjestion sec}
The proof to Theorem \ref{full congestion} is inspired by the proof to Theorem \ref{SODA full thm}.
The proof to Theorem \ref{SODA full thm} considers an isometric embedding of the graph $G$ in a discretization of its injective hull $H(G)$ called the ``Hellification'' of $G$.
The center of the core they construct is a point in $G$ that is closest to the point $m'$ in $H(G)$ that minimizes $\sum_{u \in V}d_{H(G)}(m',u)$.
The construction of $H(G)$ is only described for unweighted graphs, and it is not clear which properties of $H(G)$ will hold in more general settings.
We modify their proof by showing that the vertex $m$ in $G$ that minimizes $\sum_{u \in V}d(m,u)$ is the center of a core.
The approach is similar, but we replace the argument using fibers and the eponymous Helly property with methods derived from ``halfspaces of hyperbolic spaces'' as developed in \cite{Y2}.
Because $m$ can be found by performing Dijkstra's algorithm $|V|$ times, this leads to a faster algorithm to find a core.
While there is a bound on $d_{H(G)}(m',V(G))$, there is no known bound on $d_{H(G)}(m',m)$ and so the proofs are independent.

Similar to the proof of Theorem \ref{SODA full thm}, we will make use of different versions of hyperbolicity.
Unfortunately, while those different versions have been proven to be roughly equivalent in unweighted graphs, there is no known analogous relationship in weighted graphs.
Hence, we devote the next subsection to establishing such connections.

\subsection{Alternative Forms of Discrete Hyperbolicity}
To simplify notation in this subsection, let $a,b,c$ be an arbitrary triple of vertices in $G$.

A graph has $k$\emph{-slim triangles} if for every triple of vertices $a,b,c$ we have $P_{b,c} \subseteq B_{k}(P_{a,b} \cup P_{a,c})$.
We let $\widehat{\delta}(G)$ denote the smallest $k$ such that this property holds, and when there is no chance of confusion we will simply use $\widehat{\delta}$.
The invariants $\delta$ and $\widehat{\delta}$ are closely related.
It has been shown that $\widehat{\delta} \leq 3 \delta$ for geodesic spaces (e.g. see \cite{ABCFLMSS,BP}).
For weighted graphs, any bound on $\widehat{\delta}$ using $\delta$ must also incorporate some function of $W$.
To see why, let $0 < \epsilon \ll W$ and consider the weighted four cycle with $V = \{v_1, v_2, v_3, v_4\}$, $E = \{v_1v_2, v_2v_3, v_3v_4, v_4v_1\}$, $w(v_1v_2) = w(v_2v_3) = W/2$, $w(v_1v_4) = W$, and $w(v_3v_4) = \epsilon$ for small $\epsilon > 0$.
A direct calculation gives that $\delta(G) = \epsilon/2$ and $\widehat{\delta}(G) = W/2$.
We give new results that account for $W$ in two ways.
Lemma \ref{corner of triangle near opposing geodesic} and the following proposition are a generalization of an argument in \cite{ABCFLMSS,BP}, and Lemma \ref{thin triangle analogue} is a generalization of an argument in \cite{GH} for the related notion of thin triangles.
Ultimately we only use Lemma \ref{thin triangle analogue} later in this paper, but we present all results as they may be applicable in future works.

To give intuition for the usage of Gromov's product in the following statements, the motivation is to differentiate $P_{a,b}$ from $P_{a,c}$.
To see how this works observe that $d(b,c) = (a,c)_b + (a,b)_c$.
Thus we can partition $P_{b,c}$ based on vertices within distance $(a,c)_b$ of $b$ and those that are within distance $(a,b)_c$ of $c$.
Based on the intuition of a distance approximating tree, we can refine the formula $P_{b,c} \subseteq B_{k}(P_{a,b} \cup P_{a,c})$ to say that most of the vertices in the first category are inside $B_k(P_{a,b})$ while most of the second category are inside $B_k(P_{a,c})$.

For $y \in P_{b,c}$ define $t_{y:a,b,c} = d(b,y) - (a,c)_b$.
When we try to prove the $k$-slim triangles property holds: if $t_{y:a,b,c}$ is positive, then we will try to prove that $y$ is close to $P_{a,c}$; otherwise we will try to prove $y$ is close to $P_{a,b}$.
It is useful to first consider vertices where $t_{y:a,b,c}$ is close to zero.
By the triangle inequality, every vertex in $P_{b,c}$ will be at least $(b,c)_a$ away from $a$.
In the following statement we will show that some vertex will be close to this bound in a hyperbolic graph.
By the intuition of a distance approximating tree, it is natural to look for such a vertex at the split where $t_{y:a,b,c}$ is near zero.

\begin{lemma}\label{corner of triangle near opposing geodesic}
If $y_*$ is the vertex in $P_{b,c}$ that minimizes $|t_{y:a,b,c}|$, then $d(a,y_*) \leq 2\delta(G) + W/2 + (b,c)_a$.
\end{lemma}
\begin{proof}
Apply (\ref{4 point condition}) to $a,b,c,y_*$ to observe that (the second line uses that $y_* \in P_{b,c}$ implies $d(c,y_*) = d(b,c)-d(b,y_*)$; the third line uses arithmetic; and the fourth line uses the definition of $t_{y_*:a,b,c}$) 
\begin{eqnarray*}
d(a,y_*)	& \leq & 2 \delta - d(b,c) + \max\{d(b,y_*) + d(a,c), d(c,y_*) + d(a,b) \}\\
			& = & 2 \delta + \max\{d(b,y_*) + d(a,c) - d(b,c), -d(b,y_*) + d(a,b) \}\\
			& = & 2 \delta + \max\{d(b, y_*) + (b,c)_a - (a,c)_b, -d(b,y_*)  + (b,c)_a + (a,c)_b\} \\
			& = & 2 \delta + (b,c)_a + \max\{t_{y_*:a,b,c}, -t_{y_*:a,b,c}\} \\
			& = & 2 \delta + (b,c)_a + |t_{y_*:a,b,c}|. \\
\end{eqnarray*}
The proof concludes by observing from the definition of $W$ and $y_*$ that $|t_{y_*:a,b,c}| \leq W/2$.
\end{proof}

\begin{proposition}
For any graph we have that $\widehat{\delta}(G) \leq 3\delta(G) + W/2$.
Moreover, if $w \in P_{a,b}$ and $d(a,w) \leq (b,c)_a$, then there is a point $y \in P_{a,c}$ such that $d(w,y) \leq 3\delta(G) + W/2$ and $|d(a,y) - d(a,w)| \leq \delta(G) + W/2$.
\end{proposition}
\begin{proof}
Let $a,b,c$ be an arbitrary triple of vertices in $G$.
Let $w \in P_{a,b}$; to prove the lemma we will demonstrate the existence of $y \in P_{a,c} \cup P_{b,c}$ such that $d(w,y) \leq 3\delta + W/2$.

Because $w \in P_{a,b}$, a direct calculation gives that $(a,b)_w = 0$.
Applying hyperbolicity we have $(a,b)_w \geq \min\{(a,c)_w, (b,c)_w\} - \delta$; by symmetry we assume that $(a,c)_w \leq \delta$.
The lemma then follows from applying Claim \ref{corner of triangle near opposing geodesic} to $w,a,c$ to find suitable $y$.

By simple calculation we have that $d(a,w) \leq (b,c)_a$ is equivalent to $(a,c)_w \leq (b,c)_w$.
So, to prove the ``moreover'' part of the lemma, we must show that $|d(a,y) - d(a,w)| \leq \delta + W/2$.
This inequality follows from $d(a,y) = t_{y:a,b,c} + (c,w)_a = t_{y:a,b,c} + d(a,w) - (a,c)_w$ and that $0 \leq (a,c)_w \leq \delta$.
\end{proof}

\begin{lemma}\label{thin triangle analogue}
Fix arbitrary vertices $a,b,c$ in $G$, and let $x \in P_{a,b}$ such that $d(x , a) \leq (b,c)_a$.
Let $y \in P_{a,c}$ such that $|d(a,y) - d(a,x)|$ is minimized.
Then $|d(a,y) - d(a,x)| \leq W_{/2}$ and $d(x,y) \leq 4 \delta(G) + W_{/2}$.
\end{lemma}
\begin{proof}
That $|d(a,y) - d(a,x)| \leq W_{/2}$ follows from the definition of $W_{/2}$.
Because $x \in P_{a,b}$ we have $(x,b)_a = d(a,x)$ and symmetrically $(c,y)_a = d(a,y)$.
\begin{eqnarray*}
(x,y)_a & \geq & -\delta + \min\{(x,b)_a, (b,y)_a\} \\
	& = & -\delta + \min\{d(x,a), (b,y)_a\} \\
	& \geq & \min\{ d(x,a) - \delta, (b,c)_a - 2\delta, (c,y)_a - 2\delta\} \\
	& \geq & \min \{d(x,a) , d(a,y) \} - 2 \delta. \\
\end{eqnarray*}
We then have 
$$d(x,y) = d(x,a) + d(y,a) - 2(x,y)_a \leq |d(x,a) - d(y,a)| + 4 \delta \leq W_{/2} + 4\delta.$$ 
\end{proof}

\subsection{Congestion}

Recall that $m$ is the vertex in $G$ that minimizes $\sum_{u \in V}d(m,u)$.
Let us give intuition for the next proof.
We consider an arbitrary vertex $w$, and a vertex $w' \in P_{m,w}$ that is sufficiently far from $m$.
We consider another arbitrary vertex, $v$, and show that if $(v,w)_m \geq d(m,w')$, then $v$ is substantially closer to $w'$ than to $m$.
By definition of $m$, this can not happen for too many choices of $v$.
To see why such a statement is useful, consider the intuition of a distance approximating tree to see that if $(v,w)_m$ is not too large, then $P_{v,w}$ will contribute to the core at $m$.

\begin{proposition}\label{centroid splits into bounded components}
Let $0 < \alpha < 1/2$.
Suppose that $w', w \in V$ such that $w' \in P_{m,w}$ and $d(m,w') \geq \frac{1-\alpha}{1-2\alpha}(4 \delta(G) + 2 W_{/2})$.
Let $S = \{u \in V : (u,w)_m < d(m,w')\}$.
Under these circumstances, $|S| \geq \alpha n$.
\end{proposition}
\begin{proof}
For any vertex $u$ we have by the triangle inequality $|d(u,m) - d(u,w')| \leq d(m,w')$.
Suppose $v \notin S$ so that $(v,w)_m \geq d(m,w')$.
By Lemma \ref{thin triangle analogue} there exists a $v' \in P_{m,v}$ such that $|d(m,v') - d(m,w')| \leq W_{/2}$ and $d(v',w') \leq 4\delta + W_{/2}$.
It follows
\begin{eqnarray*}
d(v,w')	& \leq & d(v,v') + d(v',w') \\
	& = & d(v,m) - d(m,v') + d(v',w') \\
	& \leq & d(v,m) - (d(m,w')-W_{/2}) + (4\delta + W_{/2}).
\end{eqnarray*}
We first apply the the definition of $m$ and then apply the above inequality to see that 
\begin{eqnarray*}
 0  & \leq & \sum_{u \in V} d(w',u) - \sum_{u \in V} d(m,u) \\
	& = &	\sum_{u \in S} (d(w',u)-d(m,u)) + \sum_{v \notin S}(d(w',v)-d(m,v)) \\
	& \leq & |S| d(m,w') + (n-|S|)(4\delta + 2W_{/2} - d(m,w')),
\end{eqnarray*}
and therefore $|S| \geq n\frac{d(m,w') - 4\delta - 2W_{/2}}{2d(m,w') - 4\delta - 2W_{/2}}$.
The proposition follows from the bound on $d(m,w')$ in the statement of the theorem.
\end{proof}

\begin{proof}[Proof of Theorem \ref{full congestion}]
To simplify notation, we will use $r$ to denote $r(\alpha)$.
We will first prove that for every vertex $w$ there exists a set $S_w \subseteq V$ such that $|S_w| \geq \alpha n$ and for each $u \in S_w$ we have that $(w,u)_m \leq r'$ where $r' = W + \frac{1-\alpha}{1-2\alpha}(4 \delta(G) + 2 W_{/2})$.
The second part of the proof is to show that if $u \in S_w$, then $P(u,w) = P(u,w;B_r(m))$, which will imply the theorem.

By the triangle inequality $(a,b)_c \leq d(c,a)$ for any triples $a,b,c$, so if $d(m,w) \leq r'$, then we can simply set $S_w = V$.
Therefore assume that $d(m,w) > r'$.
By the definition of $W$, there exists a vertex $w' \in P_{m,w}$ such that $r'-W \leq d(m,w') \leq r'$ (possibly $w' = w$).
We let $S_w$ to be the set $S$ from Proposition \ref{centroid splits into bounded components}.
This establishes the first part of the proof.

Now suppose that $u \in S_w$, and we will show that $P_{u,w}$ contains a vertex inside $B_r(m)$.
Apply Lemma \ref{corner of triangle near opposing geodesic} to $m,w,u$ to find $y \in P_{w,u}$ such that 
$$d(m,y) \leq 2\delta + W/2 + (u,w)_m \leq 2\delta + W/2 + r' \leq r.$$
\end{proof}

\section{Algorithms that Approximate Tree-like Structure}\label{distance approximating tree sec}

\subsection{Historical Review}\label{History of tree-like algorithms subsec}

The following results are for unweighted graphs only.
Dourisboure and Gavoille \cite{DG} showed that $\tl(G)$ could be approximated to within a constant factor using something called a \emph{layering partition}, which at the time was connected to the chordality of a graph \cite{CD}.
B{\u a}doiu, Indyk, and Sidiropoulos \cite{BIS} found a method to build a distance approximating tree that approximates $\gamma_{\mult}(G)$ up to a constant factor using something called a \emph{tree-like decomposition}.
Chepoi, Dragan, Newman, Rabinovich, and Vax\`{e}s \cite{CDNRV} discovered the similarity between a tree-like decomposition and a layering partition (in fact, they are exactly the same for unweighted graphs, but there is a small difference in how the two papers transform the partition into a distance approximating tree), which led to improved bounds on the approximation of $\gamma_{\mult}(G)$.
Soon after Dragan \cite{Dr} would connect the result about $\gamma_{\mult}(G)$ and layering partitions back to Dourisboure and Gavoille's work about $\tl(G)$.
It is folklore that $\gamma_{\mult}(G) \leq 1 + \gamma_{\add}(G)$.
When the full version \cite{CDNRV2} of the conference proceedings \cite{CDNRV} appeared, it would contain additional results about layering partitions, including that $\gamma_{\add}(G) \leq O(\gamma_{\mult}(G))$.
This completes the proof to Theorem \ref{all the same for unweighted}.

For weighted graphs the story becomes more complex.
We consider four algorithms, which come from a rich diversity of motivations.
The first is by Buneman \cite{B1,B2} as he studied in the 1970's the phylogeny of \underline{Canterbury Tales} manuscripts.
Next is the approach of Gromov \cite{G} in the 1980's to construct a distance approximating tree of finite subsets of a hyperbolic group, where $\|G -T\|_\infty$ is bounded by a function of $|V|$ and $\delta(G)$.
The third algorithm is by Agarwala,  Bafna, Farach, Paterson, and Thorup's work\footnote{It has been commented that this result has been written in a manner to accept as a black box any routine that can find a best-approximation ultra-metric, and therefore it is a meta-algorithm with a ``continuum'' of possibilities.  However, in that paper they specifically describe a single, explicit algorithm for the problem, which we interpret to be ``the'' algorithm of theirs.  Other authors \cite{DHHKMW} have made the same choice.} \cite{ABFPT} on the phylogeny of biological specimens in the 1990's, which is a polynomial-time constant-factor approximation of $\gamma_{\add}(G)$. 
Finally, we will consider the previously mentioned tree-like decompositions 
of B{\u a}doiu, Indyk, and Sidiropoulos \cite{BIS} motivated in part by a question from computational geometry, which when combined with a follow-on procedure (that is outside the scope of this work) is the current-best polynomial-time method to approximate $\gamma_{\mult}(G)$.

A central aim of this paper is to draw connections between these algorithms as in the case for unweighted graphs.
We will show that---in some fundamental sense---the above four algorithms are the same.
That they are the same in a ``fundamental sense'' is not the same thing as being exactly the same; for example Buneman's algorithms are the only ones to not use an (arbitrarily chosen) root vertex.
We will define what it means for two algorithms to be the same in a ``fundamental sense'' by example, as one of these connections has already been made: Dress, Holland, Huber, Koolen, Moulton, and Weyer-Menkhoff \cite{DHHKMW} stated that Gromov's algorithm is the same as Agarwala et al.'s algorithm. 
They attribute this observation to others, but we can find no other works---before or after---that make this observation.
This connection is not at all obvious, as Gromov and Agarwala et al. use very different language in the description of their algorithms and other works have compared them side-by-side (\cite{CF2}, page 609) without coming to the same conclusion (not to mention that Gromov's algorithm, as written, requires at least $n!$ computations).
Our concept of fundamentally similar algorithms is grounded in the idea of an instructor of a programming course determining if two homework assignments are copies of each other: if we strip away cosmetic details like variable names, are the two programs performing the same operations?
For Gromov and Agarwala et al., the answer is yes.
In the third paragraph of Section 5 of \cite{DHHKMW}, Gromov's algorithm is described in a manner that directly lines up with the procedures of Agarwala et al.
For a more elaborate comparison of the two algorithms see \cite{Y2}, where we construct an explicit bijection from the operations of Agarwala et al.'s algorithm to the steps of Gromov's algorithm.

If we relax our notion of ``cosmetic details'' to ``bounded modifications,'' then connections to the other algorithms can be made, which will be made formal in this paper and is the heart of the proof to Theorem \ref{connecting the parameters thm}.
In some cases, this requires extensive analysis.
For example, Gromov merges parts of the graph based on the solution to some optimization problem, while B{\u a}doiu, Indyk, and Sidiropoulos merge parts of the graph based on the existence of paths with certain properties.
We will prove that a solution to Gromov's optimization problem is (up to a bound described by Theorem \ref{gromov is layering tree them} and the discussion before it) \emph{the special paths described by B{\u a}doiu, Indyk, and Sidiropoulos!}
Proving this connection is far from trivial, which explains why it has been missed before.
This is what we mean by two algorithms being the same in a fundamental sense: although they use different notation---an optimization problem versus special paths---they seek the same structures and then use them in the same manner.

\subsection{One Algorithm to Unite Them All}
The following algorithm is our amalgamation of tree-like decompositions from \cite{BIS} with Gromov's algorithm \cite{G}.

\begin{definition}[merged BFS]\label{defn merged BFS}
Let $P_{a,b}, P_{a,c}$ be two paths with common endpoint $a$ such that $P_{a,b} \cup P_{a,c}$ is a tree.
Let $S_x$ denote the vertices of $P_{a,x}$.
Let $w_1, w_2, \ldots, w_k$ denote an ordering of $S_b \cup S_c$ such that when $i < j$ we have $d(a,w_i) \leq d(a,w_j)$.
The process of \emph{merging paths $P_{a,b}$ and $P_{a,c}$} is to delete the edges in $P_{a,b}$ and $P_{a,c}$ and replace with with edges $w_iw_{i+1}$ for $1 \leq i < k$ with weights $w(w_iw_{i+1}) = d(a,w_{i+1}) - d(a,w_i)$.

The process to construct a \emph{merged BFS} $T$ with root $r \in V$ is as follows.
Construct a breadth-first-search tree $T'$ with root $r$.
Construct $T$ from $T'$ by merging paths $P_{r,u}$ and $P_{r,v}$ for each relevant edge $uv$ not in $T'$.
\end{definition}

We will give a deeper treatment of Gromov's algorithm \cite{G} in Section \ref{hyperbolicity subsec}.
We will give a deeper treatment of B{\u a}doiu, Indyk, and Sidiropoulos' algorithm \cite{BIS} in Section \ref{mult dist subsec}.

In what follows, we assume that $r$ is fixed.
Recall that $P_{a,b}$ denotes a shortest path in $G$ with endpoints $a$ and $b$; let $T[a,b]$ denote the unique path in $T$ with endpoints $a$ and $b$.
We define $\Delta_G = \max\{d_G(u,v):u \in T[r,v], d_T(u,v) \leq W\}$.
We make the following statements about merged BFS $T$ from graph $G$.
The proofs are short, and we only provide key details (the arguments are also standard, as similar statements are made in \cite{DG,BIS,CDNRV}).
\begin{enumerate}
	\item For every $u \in V$ we have $d_G(r,u) = d_T(r,u)$ (it suffices to consider $|E(G) \setminus E(T')| = 1$).
	\item If $uv$ is a relevant edge, then $u \in T[r,v]$ or $v \in T[r,u]$.
	\item The order of pairs of paths being merged does not affect the construction of $T$ (it suffices to consider $|E(G) \setminus E(T')| = 2$).  Moreover, while there can be different BFS $T'$ possible, the merged BFS $T$ is unique (for two BFS $T', T''$ induct on $E(T') \setminus E(T'')$, using that each edge $E(T') \setminus E(T'') \subseteq E(T')$ is relevant).
	\item For every pair of vertices $a,b$ we have $d_T(a,b) \leq d_G(a,b)$ (suffices to prove it for relevant edges).
	\item If $u \in P_{r,a}$, then $d_G(u,a) = d_T(u,a)$.
	\item We have $a \in T[r,b]$ if and only if there is a path $P$ (not necessarily shortest) of relevant edges in $G$ from $a$ to $b$ such that every vertex $v \in P$ satisfies $d(r,v) \geq d(r,a)$ (induct on $E(G) \setminus E(T)$).  
	\item Fix $a,b \in V$ and $\alpha \in \mathbb{R}$.  Every vertex $v \in T[a,b]$ satisfies $\alpha \leq d(r,v)$ if and only if there exists a path $P$ in $G$ with endpoints $a,b$ such that for every vertex $v \in P$ we have $\alpha \leq d(r,v)$ (apply the previous statement twice).
	\item For every pair of vertices $a,b$ we have $d_G(a,b) \leq d_T(a,b) + 2 W + \Delta_G$. To prove this, let $q$ denote the vertex of degree $3$ in the subgraph $T[a,r] \cup T[b,r]$ if it exists, otherwise $q = a$ where $d(r,a) \leq d(r,b)$.  Let $a' \in P_{a,r}, b' \in P_{b,r}$ be such that $d(a',r), d(b',r) \in [d(q,r)-W, \leq d(q,r)]$.  Observe that 
$$ d_G(a,b) \leq d_G(a,a') + d_G(b,b') + d_G(a',b') \leq$$
$$ \leq d_T(a,a') + d_T(b,b') + \Delta_G \leq d_T(a,b) + 2W + \Delta_G. $$
	\item $\Delta_G - W \leq \|T - G\|_\infty \leq \Delta_G + 2W$.
	\item We define  $D_\square$ to be the largest value of $\min \{d(u,v): u \in P_i, v \in P_{i+2}, i \in \{1,2\}\}$ for paths $P_1, P_2, P_3, P_4$ such that $P_j \cap P_{j+1} \neq \emptyset$ (where $j$ is modulo $4$).\footnote{Bounding $D_\square$ is an unnamed condition (e.g. see Lemma 2.1 in \cite{BIS}) that we informally call the ``slim rectangles'' condition as it is stronger but conceptually similar to the slim triangles condition (it suffices that each $P_i$ is a path).}    We have that $D_\square \geq \frac{\Delta_G}3 - \frac{3W}2$.  To see why, let $a \in T[r,b]$ such that $d_G(a,b) = \Delta_G$.  
Let $a' \in P_{r,a}$ such that $|d(a',a) - \Delta_G/3| \leq W/2$ and symmetrically define $b' \in P_{r,b}$.  Let $P_1 = P_{a,a'}$, $P_2 = P_{a',r} \cup P_{r,b'}$, $P_3 = P_{b',b}$, and let $P_4$ be a path in $G$ from $a$ to $b$ such that every vertex is at least $d(a,r)$ distance from $r$.  The bound then comes from 
$$ \min \{d(u,v): u \in P_1, v \in P_{3}\} \geq d(a,b) - d(a,u) - d(b,v) \geq \Delta_G - 2(\Delta_G/3 + W/2) = \frac{\Delta_G}3 - W$$
and
$$ \min \{d(u,v): u \in P_2, v \in P_{4}\} \geq d(v,r) - d(u,r) \geq d(a,r) - \max\{d(r,a'),d(r,b')\} \geq \frac{\Delta_G}3 - \frac{3W}2.$$
\end{enumerate}

In the next four subsections we will show how each of the four described algorithms is a slight modification of merged BFS.
In this approach we do not prove Theorem \ref{connecting the parameters thm} by directly comparing elements of $\{\tl(G), \gamma_{\mult}(G), \gamma_{\add}(G)\}$ to each other, but instead indirectly by comparing them to $\Delta_G$ and $D_\square$.
Theorem \ref{connecting the parameters thm} will follow from Theorem \ref{connecting tree length thm}, Corollary \ref{connecting tree length thm 2},  Corollary \ref{connecting add distortion thm}, Theorem \ref{connecting mult distortion thm}, Corollary \ref{connecting mult distortion thm 2}, and Remark \ref{multiply by a common factor}.

\subsection{Tree Length}\label{tree-length subsec}
We are not aware of any work that considers tree-length in the presence of edge weights.
The following theorem is a generalization of the approach of Dourisboure and Gavoille \cite{DG} by swapping out layering partitions for merged BFS.

\begin{theorem}\label{connecting tree length thm}
If $G$ has no irrelevant edges, then $\tl(G) \leq \Delta_G$.
\end{theorem}
\begin{proof}
Our proof is constructive: we will describe a tree decomposition $(T_*,\phi)$ that satisfies $\tl(T_*, \phi) \leq \Delta_G$.

Let $T_*$ be the unweighted graph that underlies the merged BFS $T$.
That is, let the edge set and vertex set of $T_*$ be exactly as $T$, but we drop the weight function $w$.
We define the function $\phi:V \rightarrow S(T_*)$ as $\phi(a) = \{b: a \in T[b,r], d_T(b,r) \leq d_T(a,r) + W\}$.
In other words, $\phi(a)$ is the replica in $T_*$ of the subtree in $T$ that is rooted at $a$ and has height $W$.

Because $T$ is a tree, if $\phi(x) \cap \phi(y) \neq \emptyset$ then $x \in T[r,y]$ or $y \in T[r,x]$ and $d_T(x,y) \leq W$.
By definition of $\Delta_G$ we have $\tl(T_*, \phi) \leq \Delta_G$.
\end{proof}

\begin{theorem}\label{tree length to thin rectangles}
For any graph, $\tl(G) \geq D_\square$.
\end{theorem}
\begin{proof}
Dourisboure and Gavoille (\cite{DG}, Lemma 5) proved something stronger than: for any tree decomposition $(T, \phi)$ and connected sets $A,B$ in $G$ there exists a node $X$ in $T$ such that $\phi^{-1}(X)$ (1) intersects each of $A$ and $B$ or (2) separates $A$ from $B$.
Let $P_1, P_2, P_3, P_4$ be as in the definition of $D_\square$.
Now apply Dourisboure and Gavoille's result with $A = P_1$ and $B = P_3$.
In case (1) $\phi^{-1}(X)$ contains an element from each of $P_1$ and $P_3$ and in case (2) $\phi^{-1}(X)$ contains an element from each of $P_2$ and $P_4$.
In either case $\phi^{-1}(X)$ contains a pair of vertices at least $D_\square$ apart.
\end{proof}

Combining Theorem \ref{tree length to thin rectangles} with the bound above relating $\Delta_G$ and $D_\square$ gives the following result.

\begin{corollary}\label{connecting tree length thm 2}
For any graph, $\tl(G) \geq \frac{\Delta_G}3 - \frac{3W}2$.
\end{corollary}

We claim to be connecting the best algorithms for tree-length, multiplicative distortion, and additive distortion, which motivates us to consider the strengths of other algorithms that approximate these parameters.
Layering partitions are only one of three algorithms presented by Dourisboure and Gavoille \cite{DG} for computing a tree-decomposition with small length in an unweighted graph.
They admit that one of the algorithms is worse than layering partitions, and we now investigate the third algorithm, which we will call \emph{$(k, \ell)$-Disk Tree}.
The algorithm progresses iteratively, let $(T_i, \phi_i)$ denote the partial tree-decomposition constructed at the end of stage $i$ (we begin with $T_0 = \emptyset$).
During stage $i$, we grow $T_{i-1}$ into $T_i$ by adding one node to the tree and staying constant on the existing nodes (in other words, $\phi_i^{-1} \vert _{V(T_{i-1})} = \phi_{i-1}^{-1}$).

We first discuss some ideas related to the algorithm to give the reader an intuition for the steps.
Afterwards we will give the details of the algorithm.

Let $H_i$ be the graph induced on the vertex set $\phi^{-1}(T_{i})$.
For partial tree decomposition $(T_i, \phi_i)$, let $\mathcal{P}_\ell$ be the property that 
\begin{itemize}
	\item if $x$ and $y$ are in $H_i$ and there exists a path from $x$ to $y$ in $G - H_i$, then $d(x,y) \leq \ell$ and
	\item for each connected component $C$ of the graph $G - H_i$ there exists a vertex $w \in V(T)$ such that $\cup_{u \in C} (N(u) \cap H_i) \subseteq \phi^{-1}(w)$.
\end{itemize} 
It can be proved that if $(T', \phi')$ is a tree decomposition and $T''$ is a subtree of $T'$, then $(T'', \phi' \vert_{V(T'')})$ satisfies the property $\mathcal{P}_{\tl(T', \phi')}$.
The design of the $(k, \ell)$-Disk Tree algorithm centers around maintaining the $\mathcal{P}_{\ell}$ property for user-provided parameter $\ell$ as nodes are added to the tree.

Let $v_i$ be the node added during state $i$ (so $\{v_i\} = V(T_i) \setminus V(T_{i-1})$).
It is well-known in the field of algorithmic results for constructing a tree decomposition that we can assume $\phi^{-1}(v_i) \setminus H_{i-1}$ is nonempty (when $i > 1$).
Suppose $y_i \in \phi^{-1}(v_i) \setminus H_{i-1}$; we may further assume that $y_i$ has a neighbor $x_i$ in $H_i$.
Those ideas give intuitive insight to the $(k, \ell)$-Disk Tree algorithm, which builds the set $\phi^{-1}(v_i)$ by first choosing $x_i$.

Formally, the algorithm performs as follows during stage $i$.
It chooses some connected component $C$ of the graph $G - H_{i-1}$, and let $x_i \in N(C) \cap H_{i-1}$ (if $i=1$, then let $x_i$ be any vertex of $G$ chosen randomly).
In $T$, the vertex $v_i$ will be a leaf appended to the vertex $w$ guaranteed to exist by the second criteria of property $\mathcal{P}_{\ell}$.
We originally set $\phi^{-1}(v_i) = (N(C) \cap H_{i-1}) \cup (B_k(x_i) - H_{i-1})$.
Note that this original set satisfies $\max\{d(u,w):u,w \in \phi^{-1}(v_i)\} \leq k + \max\{\ell, k\}$.
However, we still need to make sure that $T_i$ satisfies $\mathcal{P}_\ell$, and so elements of $B_k(x_i) - H_{i-1}$ are removed one at a time from $\phi^{-1}(v_i)$ until $\mathcal{P}_\ell$ is satisfied.
If $\phi^{-1}(v_i) \subseteq H_{i-1}$ (in other words, the algorithm failed because no vertex remains that could correspond to $y_i$), then repeat stage $i$ with a different pair $(C, x_i)$.
Otherwise add $v_i$ to the tree an proceed to stage $i+1$.

Dourisboure and Gavoille \cite{DG} prove that when $k = \ell$, each $(T_i, \phi_i)$ is a distance approximating tree of $H_i$ such that $\max_{\phi_i(x) \cap \phi_i(y) \neq \emptyset} d_G(x,y) \leq 2k$.
Hence the algorithm successfully builds a tree decomposition with small length if it terminates.
Furthermore, they showed that if $k = \ell \geq 3 \tl(G) - 2$, then there exists a $r$ such that $H_r = G$, and so $(k, \ell)$-Disk Tree is a $6$-approximation algorithm for constructing a minimum-length tree decomposition.
They believed that their algorithm was significantly stronger than the bound that they could prove, and they conjectured that with a refinement, the $(k, \ell)$-Disk Tree algorithm can be augmented into a $2$-approximation algorithm.

\begin{conjecture}[\cite{DG}, conjecture 1]\label{tree length conjecture}
Suppose that $k = \ell = \tl(G)$, and whenever there is a choice between removing $z$ or $y$ while iteratively removing vertices from $\phi^{-1}(v_i)$, we always choose to remove $z$ if $d(z, x_i) > d(y, x_i)$.
Under these conditions, there exits an $r$ where $H_r = G$.
\end{conjecture}

Unfortunately, the conjecture is not true.
In a tree decomposition each set $\phi^{-1}(w)$ is a separating set when $w$ is not a leaf in $T$.
Sometimes the separating sets of a graph form long and thin subgraphs of $G$, instead of forming successively smaller balls as is suggested in the conjecture.
This is exactly the situation in the following counterexample, which contradicts Conjecture \ref{tree length conjecture} by showing the $(k,k)$-Disk Tree algorithm is at best a $4$-approximation algorithm.

\begin{theorem} \label{the fall of the disk tree}
Let $G_{r,t}$ be the Cartesian product of a cycle $C_t$ and a path $P_{rt}$ for $r > 3$ and even $t$.  
The tree length of $G_{r,t}$ is $\frac{t}{2} + 1$, but the $(k,k)$-disk-tree algorithm fails to terminate when $k < t - 2$.
\end{theorem}
\begin{proof}
Consider the Cartesian product of a cycle and a path---that is $V(G) = \{u_{i,j}: 1\leq i \leq t, 1 \leq j \leq rt\}$ and $E(G) = \{x_{a,b}x_{a,c}:|b-c| = 1\} \cup \{x_{a,c}x_{b,c} : |b - a| \equiv 1 (mod\ t)\}$.
We see that $\tl(G) \leq t/2 + 1$ by constructing a path decomposition $P = w_1, \ldots, w_{rt - 1}$ where $\phi^{-1}(w_i) = \{x_{a,b}: i \leq b \leq i+1\}$.
On the other hand, this graph contains a subgraph $H$ that is isometric to a $ t/2 +1$ by $t/2 + 2$ Euclidean grid, whose tree-length is known to be $t/2 + 1$ (\cite{DG}, Theorem 3).
Any tree-decomposition of $G$ induces a tree-decomposition of $H$, and since the subgraph is isometric we see that the tree length of $G$ is exactly $ t/2 + 1 $.

Now suppose we try to run $(h, h)$-Disk Tree on $G$, where $h \leq t - 3$.
Consider the first ball the algorithm grabs: it has a diamond shape and radius $h$.
If this radius is large enough to separate the graph, then there exists a ``left'' and ``right'' $H_i$ (draw the copies of $C_t$ vertically and the copies of $P_{rt}$ horizontally).
However, the diamond shape is a problem: the diameter of the boundary of the ball is $2h$.
The algorithm iteratively removes points from the boundary of the ball until the diameter of the boundary is at most $h$.
But it chooses the vertex to remove based on distance from the center: so the boundary never flattens into a column, but always maintains two points that are $h/2$ apart in each coordinate.
Eventually the ball has lost so many points that it no longer contains any vertex in a copy of $P_{r\ell}$, and the two components of $G-H_i$ merge into one.
With the two components merged, the ball continues to shrink until it has radius $(h-1)/2$.
The algorithm now stalls: the next ball will shrink using the same logic until it is a subset of the original ball.
\end{proof}

\subsection{Hyperbolicity}\label{hyperbolicity subsec}
We assume $r \in V$ is fixed.

We now consider the algorithm of Agarwala,  Bafna, Farach, Paterson, and Thorup \cite{ABFPT} to construct a distance approximating tree $T$ such that $\|G - T\|_\infty \leq 6 \gamma_{\add}(G)$.
Because the connection \cite{DHHKMW} to Gromov's \cite{G} algorithm has already been made, we will use his construction as a proxy for the construction of Agarwala et al.
We sketch his argument below.
\emph{Adding a vertex by subdividing edge $uv$} is to delete edge $uv$, add new vertex $z$, and add edges $uz,zv$ with weights satisfying $w(uz) + w(zv) = w(uv)$.
By construction, adding a vertex by subdividing an edge does not change the distances between any pair of vertices.

\begin{theorem}[\cite{B1}]\label{zero curvature tree}
If $\delta(G) = 0$, then there exists a distance approximating tree $T$ such that $\|T - G\|_\infty = 0$.
\end{theorem}
\begin{proof}[Sketch of the proof in \cite{G}]
For a fixed $r \in V$, construct a tree $T'$ such that every edge is incident with $r$.
For $u \neq r$, the edge $ur$ has weight $w(ur) = d_G(ur)$.
For every pair of vertices $u,v \in V \setminus \{r\}$ with $d(u,r) \leq d(v,r)$, add a vertex $w_{u,v}$ in $T'[r,u]$ by subdividing an edge such that $d(w_{u,v},r) = (u,v)_r$, and then merge the paths $T'[w_{u,v},r]$ and $T'[v,r]$.
\end{proof}

Gromov's algorithm has two steps: first he constructs a quasi-isometry from $G$ to $G'$ such that $\delta(G') = 0$, and then Gromov applies Theorem \ref{zero curvature tree} to $G'$.
We will revisit key lemmas Gromov established about the quasi-isometry, and we will show how that construction is based on calculations that differ in a bounded amount from the calculations to construct a merged BFS.
The algorithm uses a fixed root vertex $r$ that is chosen arbitrarily.
We will repeatedly use several facts about Gromov's product that are proven via the triangle inequality.
\begin{itemize}
	\item For every relevant edge $uv$ with $d(r,u) \leq d(r,v)$ we have $(u,v)_r \geq d(r,u) - W/2$.
	\item If $d(x,y) = d(x,z) + d(z,y)$ (which would hold if $z \in P_{x,y}$), then $(x,y)_r \leq (x,z)_r$ (symmetrically $(x,y)_r \leq (z,y)_r$).
\end{itemize}

\begin{definition}\label{defn f}
For a given metric $d$ over finitely many elements, let $F:V \times V \rightarrow \mathbb{R}$ such that
$$ f(x,y) = \max_{x=w_1, w_2, \dots , w_k = y} \min_{1 \leq i \leq k-1} (w_i , w_{i+1})_r ,$$
for arbitrary $k$, and let $d'$ be a new distance function over the vertices of $G$ defined as 
$d'(x, y) = d(x,r) + d(y,r) - 2f(x,y)$.
\end{definition}

\begin{lemma}[\cite{G}]\label{describing d'}
The following are true:
\begin{enumerate}
	\item $d'$ is a pseduo-metric (in other words, it is satisfies all the properties of a metric except for allowing $d'(x,y) = 0$ when $x \neq y$).
	\item $d'(x,y) \leq d(x,y)$.
	\item $d'(x,y) \geq d(x,y) - \delta(G) \lceil \log_2(n-2) \rceil$.
	\item For all $x$ we have $d'(x,r) = d(x,r)$.
	\item $d'$ is $0$-hyperbolic.  
\end{enumerate}
\end{lemma}

Gromov does not provide an efficient manner to compute $f(x,y)$, only commenting that we can assume no element of $w_1, w_2, \dots , w_k$ repeats and therefore finitely many terms need to be considered. 
Agerwala et al. compute $f$ efficiently by approximating a distance function derived from $d$ by an ultra-metric.
We connect Gromov's algorithm to a merged BFS through the following insight.

\begin{claim} \label{paths}
For every pair of vertices $x,y$, there exists a sequence of vertices $x=w_1, \ldots w_k =y $ that maximizes the expression defining $f(x,y)$ that is also a path of relevant edges.
\end{claim}
\begin{proof}
For fixed $x,y$, suppose that $x = w_1, \ldots, w_k = y$ is an extremal sequence with $f(x,y) = \min_{1 \leq i \leq k-1} (w_i , w_{i+1})_r$.
Suppose there is an $i$ such that $w_iw_{i+1}$ is not a relevant edge, which implies the existence of a vertex $u \notin \{w_i, w_{i+1}\}$ on a shortest path from $w_i$ to $w_{i+1}$, which implies $d(w_i,u) + d(u,w_{i+1}) = d(w_i,w_{i+1})$.
By the triangle inequality, we have $d(r,u) \geq d(r,w_{i+1}) - d(w_{i+1},u)$, so 
\begin{eqnarray*}
 (w_i, u)_r 	& = & 	\frac{d(w_i,r) + d(u,r) - d(w_i,u)}2 \\
				& \geq & \frac{d(w_i,r) + (d(r,w_{i+1}) - d(w_{i+1},u)) - d(w_i,u)}2 \\
				& = &  \frac{d(w_i,r) + (d(r,w_{i+1}) - (d(w_i,w_{i+1})- d(w_i,u))) - d(w_i,u)}2 \\
				& = &  \frac{d(w_i,r) + d(r,w_{i+1}) - d(w_i,w_{i+1})}2 = (w_i, w_{i+1})_r. \\
\end{eqnarray*}
A symmetric argument gives that $(u,w_{i+1})_r \geq (w_i,w_{i+1})_r$.
Thus the sequence 
$$x = w_1, \ldots, w_i, u, w_{i+1}, \ldots , w_k = y$$
 is another extremal sequence.
Repeat this procedure until $w_1, \ldots, w_k$ forms a path of relevant edges.
\end{proof}

Let $(x , y)_r' = (d'(x, r) + d'(y, r) - d'(x, y))/2$.
That is, when Gromov applies Theorem \ref{zero curvature tree}, it is with graph $G'$ and distance function $d'$; and $(x,y)_r'$ denotes where the vertex $w_{u,v}$ is constructed.
The following is a direct calculation.

\begin{claim} \label{identifying vertices}
The following are true.
\begin{enumerate}
	\item $(x , y)_r' = f(x, y)$.
	\item $u \in T[r,v]$ if and only if there exists a path $u = w_1, \dots w_k = v$ of relevant edges such that $(w_i,w_{i+1})_r \geq d(r,u)$ for each $i$. 
	\item If $u \in P_{v,r}$, then $(u,v)_r' = d(u,r)$.
\end{enumerate}
\end{claim}

We give the following definition for a generalized merged BFS:
\begin{itemize}
	\item Construct a function $g:E\rightarrow \mathbb{R}$ where $g(u,v) \leq \min\{d(r,u), d(r,v)\}$.
	\item Construct a BFS tree $T'$.
	\item For every relevant edge $uv$ with $d(r,v) \geq d(r,u)$ add a vertex $w_{u,v}$ in $T[r,u]$ by subdividing an edge such that $d(r,w_{u,v}) = g(u,v)$.
	\item Merge paths $T[r,v]$ and $T[r,w_{u,v}]$.
\end{itemize}
To recreate Definition \ref{defn merged BFS}, set $g(u,v) = \min\{d(r,u), d(r,v)\}$.
To recreate Gromov's algorithm, set $g(u,v) = (u,v)_r$.
But these two things are not very different: if $d(u,r) = \min\{d(r,u), d(r,v)\}$, then $d(u,r) \geq (u,v)_r \geq d(u,r) - W/2$.

\begin{theorem}\label{gromov is layering tree them}
For a given graph $G$, let $T$ be the tree constructed by merged BFS and let $T_*$ be the tree constructed by Gromov's algorithm.
For all $u,v$ we have $d_T(u,v) \leq d_{T^*}(u,v) \leq d_T(u,v) + W$.
\end{theorem}
\begin{proof}
Suppose $u',u'' \in T[r,u]$ such that $d(r,u') \leq d(r,u'')$.
Let $G'$ be constructed from $G$ by merging $T[r,v]$ with $T[r,u']$, and let $G''$ be constructed by merging $T[r,v]$ with $T[r,u'']$.
By the definition of merging paths, for all $u,v$ we have $d_{G''}(u,v) \leq d_{G'}(u,v)$.
Applying this fact inductively on $E(G) \setminus E(T')$, we have that $d_T(u,v) \leq d_{T^*}(u,v)$ for all $u,v$.
We now must prove that for arbitrary $u,v$ we have $d_{T^*}(u,v) \leq d_T(u,v) + W$.

In what follows, we use function $f$ from Definition \ref{defn f}.

If $u \in T[r,v]$, then there exists a path of relevant edges in $G$ such that every vertex $z$ in the path satisfies $d(r,z) \geq d(u,r)$.
The same path establishes that $f(u,v) \geq d(u,r)-W/2$.
Therefore, by the construction given in Theorem \ref{zero curvature tree}, in $T^*$ the vertex $w_{u,v}$ satisfied $d(u,w_{u,v}) \leq W/2$ and so 
$$d_{T^*}(u,v) \leq d_{T^*}(v,w_{u,v}) +  d_{T^*}(w_{u,v},u) \leq (d_T(v,u) + W/2) + W/2. $$

Otherwise we may assume that there exists a unique vertex $x$ of degree $3$ in $T[r,u] \cup T[r,v]$.
Because $T$ is a tree, $d_T(u,v) = d_T(u,x) + d_T(x,v)$.
Recall that there exists a path of relevant edges in $G$ such that every vertex $z$ in the path satisfies $d_G(z,r) \geq d_T(r,x)$.
The same path establishes that $f(u,v) \geq d_T(x,r)-W/2$.
Therefore, by the construction given in Theorem \ref{zero curvature tree}, in $T^*$ the vertex $w_{u,v}$ satisfied $d_{T^*}(r,w_{u,v}) \geq d_T(x,r) -  W/2$ and so 
$$d_{T^*}(u,v) \leq d_{T^*}(v,w_{u,v}) +  d_{T^*}(w_{u,v},u) \leq (d_T(v,x) + W/2) + (d_T(u,x)+W/2) \leq d_T(u,v) + W.$$
\end{proof}

\begin{corollary}\label{connecting add distortion thm}
For any graph $G$ with merged BFS $T$ we have $\gamma_{\add}(G) \leq \|T - G\|_\infty \leq 6 \gamma_{\add}(G) + W$.
Moreover, $\frac{\Delta_G-2W}6 \leq \gamma_{\add}(G) \leq  \Delta_G + 2W$ .
\end{corollary}
\begin{proof}
By definition of $\gamma_{\add}$ we have $\gamma_{\add}(G) \leq \|T - G\|_\infty$.
Let $T_*$ be the tree constructed by Gromov's algorithm; Agarwala,  Bafna, Farach, Paterson, and Thorup \cite{ABFPT} proved that $\|T_* - G\|_\infty \leq \leq 6 \gamma_{\add}(G)$.
By Theorem \ref{gromov is layering tree them}, $\|T - G\|_\infty \leq \|T_* - G\|_\infty + W$, so $\|T - G\|_\infty \leq 6 \gamma_{\add}(G) + W$.

As previously stated for merged BFS $T$, we have 
$$ \Delta_G - W \leq  \|T - G\|_\infty \leq \Delta_G + 2W. $$
So the final inequality comes from combining the above inequalities to see that 
$$ \gamma_{\add}(G) \geq \frac{\|T - G\|_\infty - W}6 \geq \frac{\Delta_G  - 2W}6.$$
\end{proof}

\subsection{Multiplicative Distortion}\label{mult dist subsec}
A tree-like decomposition $L$ is a structure from \cite{BIS} and is similar to a merged BFS $T$.
The distinction is that the tree-like decomposition includes a parameter $\lambda \geq 1$ where distances $d(a,r)$ are rounded to the nearest multiple of $\lambda W$.
There is also a cosmetic change, where instead of merging one pair of paths at a time as in Definition \ref{defn merged BFS}, the authors of \cite{BIS} merge many paths at once by calculating components in a particular auxiliary graph.

Theorem \ref{connecting mult distortion thm} is a weaker result that is implied by B{\u a}doiu, Indyk, and Sidiropoulos \cite{BIS}.
For completeness sake, because of the slight differences between tree-like decompositions and merged BFS, we include a proof.
Let $W' = \frac{\max_{uv \in E}w(uv)}{\min_{uv \in E}w(uv)}$.

\begin{theorem}\label{connecting mult distortion thm}
Assume all edges are weighted at least $1$ and at most $W'$.
If $\gamma_{\mult}(G) \geq 12$, then $\gamma_{\mult}(G) \leq 6 \Delta_G$.
\end{theorem}
\begin{proof}
We construct a distance approximating tree $T$ as follows:
\begin{itemize}
	\item Create a copy $G'$ of $G$, where vertex $u$ in $G$ is copied to vertex $u'$ in $G'$.
	\item Create merged BFS $T$ over $G'$.
	\item For each $u \in V$, add $u$ to $T$ as a leaf by appending edge $u'u$ with weight $1$. Multiply all edge weights by $1/3$.  Call this tree $T^*$.
	\item Let $\gamma = \min_{u,v} d_G(u,v)/d_{T^*}(u,v)$.  Multiply all edge weights by $\gamma$.
\end{itemize}
Recall that for all vertices $d_T(u,v) \leq d_G(u,v)$, and therefore $d_{T^*}(u,v) \leq (d_G(u,v) + 2)/3 \leq d_G(u,v)$.
So $\gamma_{\mult}(G) \leq \gamma$, and the theorem will follow when we show that $\gamma \leq 6 \Delta_G$.

Let $u,v$ be such that $d_G(u,v) = \gamma d_{T^*}(u,v)$.
By definition of $T^*$ and the properties of a merged BFS we have that 
$$d_{T^*}(u,v) \geq \max\{2/3, \frac{1}{3}(d_G(u,v) - \Delta_G - 2W')\} \geq \max\{2/3, \frac{1}{3}(d_G(u,v) - 3\Delta_G)\}.$$
Thus $d_G(u,v) \geq 2 \gamma/3$ and $\Delta_G \geq (1/3 - \gamma^{-1})d_G(u,v) \geq d_G(u,v)/4$.
Therefore $\Delta_G \geq \gamma/6$.
\end{proof}

\begin{theorem}[\cite{BIS}, Lemma 2.1] \label{D square and mult}
Assume all edges are weighted at least $1$ and at most $W'$.
We have $\gamma_{\mult}(G) \geq D_\square / W'$.
\end{theorem}

Combining Theorem \ref{D square and mult} with our prior inequality between $\Delta_G$ and $D_\square$ establishes the following result.

\begin{corollary}\label{connecting mult distortion thm 2}
Assume all edges are weighted at least $1$ and at most $W'$.
We have $\gamma_{\mult}(G) \geq \frac{\Delta_G}{3W'} - 1.5$
\end{corollary}

Unfortunately, Corollary \ref{connecting mult distortion thm 2} is tight (which is why B{\u a}doiu, Indyk, and Sidiropoulos use follow-on procedures to modify their tree-like decomposition to create a bound that is independent of $W'$ but exponentially dependent on the maximum degree of the graph).
To see tightness, consider a graph $G$ with vertices $v_1, v_2, \ldots, v_{k+1}$, edges $v_1v_{k+1}$ and $v_iv_{i+1}$ for $1 \leq i \leq k$, and weights $w(v_iv_{i+1}) = 1$ and $w(v_1v_{k+1}) = k-\epsilon$ for $0 < \epsilon \ll k$.
By deleting edge $v_1v_{k+1}$ we construct a tree that proves $\gamma_{\mult}(G) \leq 1 + O(\epsilon/k)$, but the construction in Theorem \ref{connecting mult distortion thm} using a merged BFS rooted at $r = v_{\lfloor k/2 \rfloor}$ yields a tree with multiplicative distortion $\Omega(k)$.

\begin{remark}\label{multiply by a common factor}
The statements in this subsection assume that the edge weights are in the range $[1, W']$, which is accomplished by multiplying all edge weights by a common factor $W'/W$.
If we multiply every edge weight by a common factor, then $W$, $\Delta_G$, $\tl(G)$, $\gamma_{\add}(G)$, and $D_\square$ each grow proportional but $\gamma_{\mult}$ is invariant.  
Because $W$, $\Delta_G$, $\tl(G)$, $\gamma_{\add}(G)$, and $D_\square$ all grow by the same factor, that factor cancels out when we compare 
\begin{itemize}
	\item $\Delta_G$ to $\tl(G)$ or $\gamma_{\add}(G)$, or 
	\item $\Delta_G$ to $D_\square$.
\end{itemize}
The assumption $\gamma_{\mult}(G) \geq 12$ in Theorem \ref{connecting mult distortion thm} becomes $\gamma_{\mult}(G) \geq 12 W/W'$ prior to the adjustment of edge weights.
Because $W' \geq 1$, for this assumption to hold it is sufficient that $\gamma_{\mult}(G) \geq 12 W$ in the original graph without modification to edge weights.

Assume all edges are relevant.
Because we have already seen that $\tl(G) \leq \Delta_G \leq 3 \tl(G) + \frac{9W}2$ and $\gamma_{\add}(G) - 2W \leq \Delta_G \leq 6 \gamma_{\add}(G) + 2W$, it follows that 
\begin{itemize}
	\item  $\gamma_{\mult}(G) \leq 12W + \min\{18 \tl(G) + 15 W, 36 \gamma_{\add}(G) \} $      , and
	\item   $\max \left\{\frac{2\tl(G)-9W}6, \frac{2\gamma_{\add}(G) - 13W}6  \right\} \leq W\gamma_{\mult}(G)$    .
\end{itemize}
\end{remark}

\subsection{Buneman's Algorithms}


Buneman gave separate algorithms in \cite{B1} and \cite{B2} to construct a distance approximating tree of a finite $0$-hyperbolic metric space.
In both cases he proved that the resulting distance approximating tree $T$ satisfies $\|T - G\|_\infty = 0$.
He does not provide any bounds on how the algorithms perform when $\delta(G) > 0$.
Even worse, the algorithm from \cite{B2} is not well-defined in general.
It constructs an edge of a distance approximating tree $T$ with weight $\mu$ by finding a partition of the form $V = S_1 \cup S_2$ that satisfies an equation\footnote{We can compare this formula to the results in Section \ref{hyperbolicity subsec}: Gromov constructs an edge with weight $\mu$ that separates $V = S_1 \cup S_2$ with $r \in S_1$ when there is a spanning tree $T_2$ over $S_2$ such that $\mu = \min_{c'd \in E(T_2)} (c',d)_r - \max_{b \in S_1 c \in S_2} (b,c)_r$.} that is equivalent to $\mu  = \min_{a,b \in S_1, c,d \in S_2} (c,d)_a - (c,b)_a > 0$.
The algorithm finds $2n-1$ such partitions, and links the edges using a distance function over partitions.
Non-trivial partitions of this form do not always exist when $\delta(G) > 0$; for example all partitions of this form in an unweighted cycle satisfy $|S_1| \in \{1,n-1\}$.

The final step in unifying many of the leading algorithms to construct a distance approximating tree---that is, to show how a particularly clever technique has bubbled out of many different contexts independently, as described in the Introduction---is to show that the algorithm from \cite{B1} after small modifications is the same algorithm described by Gromov. 

Buneman's algorithm from \cite{B1} begins by finding a triple $r,x,y$ that maximizes $(x,y)_r$.
Buneman then constructs a tree $T_{r,x,y}$ by recursing on a graph with one fewer vertices by deleting $x,y$ and adding vertex $w_{x,y}$ where $d(u,w_{x,y}) = d(x,u) - (r,y)_x$ for each $u \in V \setminus \{x,y\}$\footnote{Buneman proves that if $\delta(G) = 0$, then by the choice of $x,y,r$ we have that $d(x,u) - (r,y)_x = d(y,u) - (r,x)_y$, and so $x$ is symmetric with $y$.} (Buneman works with metric spaces and so implicitly adds a weighted edge from $w_{x,y}$ to each other vertex, but it can be seen that it suffices to only add weighted edges between $w_{x,y}$ and $N(x) \cup N(y)$).
From $T_{r,x,y}$ we construct $T$ by attaching $x$ and $y$ as leaves adjacent to $w_{x,y}$ incident to edges with weight $(r,y)_x$ and $(r,x)_y$ respectively.

A connection between Buneman's algorithm and Gromov's algorithm can be made immediately.
If the vertex $r$ chosen by Buneman coincides with the arbitrarily chosen root $r$ in Gromov's algorithm, then by the choice of $r,x,y$, we have for $f$ from Definition \ref{defn f} that $f(x,y) = (x,y)_r$, which by arithmetic implies that $d(w_{x,y},x) = (r,y)_x$ and $d(w_{x,y},y) = (r,x)_y$.
Thus, $w_{x,y}$ from Gromov's construction in Theorem \ref{zero curvature tree} is exactly as described for Buneman's algorithm, and all other vertices added by Gromov's algorithm are at least as far from $x$ and $y$ as $w_{x,y}$.
Therefore Gromov constructs vertex $w_{x,y}$ and edges $w_{x,y}x, w_{x,y}y$ in the same manner as in \cite{B1}.

There are two key differences between the algorithms.
The first is that Gromov uses a single fixed root $r$ for the entirety of the algorithm, while the ``appropriate root'' may change when Buneman's algorithm recurses on a smaller graph.
The other difference is that the introduction of $w_{x,y}$ in $T_{r,x,y}$ may change the underlying metric space, and so the algorithm may behave differently when Buneman's algorithm recurses on a smaller graph.
We will show that both of these discrepancies can be extinguished without substantial modification to Buneman's approach, which proves that the two algorithms are the same in a fundamental sense (if not perfectly identical in a technical sense).

First, it can be easily seen that choosing the triple $r,x,y$ that maximizes $(x,y)_r$ is stronger than necessary for Buneman's proof of the $\delta(G) = 0$ case; the entire argument still holds if we assume $r$ is fixed and choose $x,y$ to maximize $(x,y)_r$.
This handles the first key difference.
We present the following algorithm as a slight modification of Buneman's algorithm; we are only changing the construction of $T_{r,x,y}$ in respect to the fact that $x$ and $y$ are not handled symmetrically in \cite{B1} when $\delta(G) > 0$.
For elegance of presentation, we also describe the construction of $T_{r,x,y}$ using the addition/deletion of edges in a graph rather than the modification of a finite metric space.
\begin{itemize}
	\item[] \textbf{Modified Buneman's Distance Approximating Tree.} Input: $(G, r)$ for $r \in V(G)$.  Output: a distance approximating tree $T$.
	\item Find pair $x,y \in V$ that maximizes $(x,y)_r$.  If there is a tie, choose the pair that maximizes $d(r,x) + d(r,y)$.
	\item If $r \in \{x,y\}$, then $(x,y)_r = 0$.  By choice of $x,y$ we have $(u,v)_r = 0$ for all $u,v \in V$.  Return the following tree as $T$: construct an edge from $r$ to each of $u \in V \setminus \{r\}$ with weight $w(ur) = d(u,r)$.  Observe $\|G-T\|_\infty = 0$. The algorithm is finished.
	\item Otherwise, construct $G_{x,y}$ as follows. 
	\begin{itemize}
		\item Remove irrelevant edges.
		\item Add vertex $w_{x,y}$ and an edge from $w_{x,y}$ to each vertex in $N(x) \cup N(y)$.
		\item For each $u \in N(x) \setminus N(y)$, set $w(u w_{x,y}) = w(ux) - (r,y)_x$. For each $v \in N(y) \setminus N(x)$, set $w(v w_{x,y}) = w(vy) - (r,x)_y$.
		\item For $t \in N(x) \cap N(y)$, set $w(t w_{x,y}) = \min\{w(tx) - (r,y)_x,  w(ty) - (r,x)_y\}$.
		\item Delete $x$, $y$, and all incident edges.
	\end{itemize}
	\item Recurse on $(G_{x,y},r)$ to produce distance approximating tree $T_{x,y}$.
	\item Construct $T$ from $T_{x,y}$ by adding vertices $x$ and $y$ and edges $w_{x,y}x$ and $w_{x,y}y$ with weights $(r,y)_x$ and $(r,x)_y$, respectively.	
\end{itemize}

First we will prove that this algorithm is well-defined with nonnegative edge weights even when $\delta(G) > 0$, and then we will prove the equivalence with Gromov's algorithm.

\begin{proposition}
For a fixed $r$ in $G$, if $x,y \in V$ maximize $(x,y)_r$ then for every $u \in N(x)$ we have $w(ux) \geq (r,y)_x$.
\end{proposition}
\begin{proof}
By arithmetic, we have that $(x,y)_r = d(r,x) - (r,y)_x$.
For any vertex $u$, by the triangle inequality we have that $(x,u)_r \geq d(r,x) - d(x,u)$.
By choice of $x,y$ we have $(x,u)_r \leq (x,y)_r$, and so $d(x,u) \geq (r,y)_x$.
Regardless of whether the edge is relevant or not, if $u \in N(x)$, then $w(ux) \geq d(x,u)$.
\end{proof}

\begin{theorem}
For any weighted graph, the distance approximating tree produced by Modified Buneman's Distance Approximating Tree is the same tree produced by Gromov's algorithm (as described in Section \ref{hyperbolicity subsec}) when run with the same root $r$.
\end{theorem}
\begin{proof}
Let $(G,r)$ be given, and let $x,y$ be chosen as in Modified Buneman's Distance Approximating Tree.
Let $T_B$ be the tree constructed by Modified Buneman's Distance Approximating Tree, and let $T_G$ be the tree constructed by Gromov's algorithm.
If $r \in \{x,y\}$, then it is clear that the two algorithms perform the same operations; so assume $r \notin \{x,y\}$.
As discussed above, both algorithms will create a vertex $w_{x,y}$, and in each of $T_B$ and $T_G$ the vertices $x$ and $y$ are leaves appended to $w_{x,y}$ with edges of weight $(r,y)_x$ and $(r,x)_y$, respectively.
Let $T_B' = T_B \setminus \{x,y\}$ and $T_G' = T_G \setminus \{x,y\}$; the theorem will follow when we prove that $T_B'$ is the same as $T_G'$.
Let $G_{x,y}$ be the graph that is recursed on by Modified Buneman's Distance Approximating Tree.
Observe that $G_{x,y}$ has one fewer vertices than $G$.
To prove that $T_B'$ is $T_G'$, by induction it is equivalent to prove that Gromov's algorithm run on $G_{x,y}$ will generate $T_G'$ (in other words, we must prove that our formulation of Modified Buneman's Distance Approximating Tree has extinguished the second key difference as described above).

As we are dealing with several graphs at once ($G$ and $G_{x,y}$ and others), we will use subscripts $d_H$ and $w_H$ to clarify that a distance or edge weight, respectively, is taken in the graph $H$.
We will also use subscripts to clarify the function $f$ from Definition \ref{defn f} and with Gromov's product: $(a,b)_{c,H} = \frac{d_H(a,c) + d_H(b,c) - d_H(a,b)}2$.

An important and subtle detail is that the introduction of $w_{x,y}$ \emph{will} affect the underling distance metric.
In particular, for any pair $u,v$ such that $P_{u,v}$ contains both $x$ and $y$, we have $d_G(u,v) > d_{G_{x,y}}(u,v)$.
What we will show is that somehow these changes are not relevant to the procedures of Gromov's algorithm.
Recall that Gromov's algorithm maps $G_{x,y}$ to $G'$ such that $\delta(G') = 0$ and then performs the operation described by Theorem \ref{zero curvature tree} on $G'$.
The operations in Theorem \ref{zero curvature tree} are only dependent on the values $d_{G'}(r,u)$ and $(u,v)_{r,G'}$ for $u,v \in V$.
By Lemma \ref{describing d'} we have $d_{G'}(r,u) = d_{G_{x,y}}(r,u)$ and by Claim \ref{identifying vertices} we have $(u,v)_{r,G'} = f_{G_{x,y}}(u,v)$. 
So, rather than proving that the underlying metric remained invariant under the addition of $w_{x,y}$ (which is not true), we will prove the weaker statement that $d_{G_{x,y}}(r,u) = d_G(r,u)$ and $f_{G_{x,y}}(u,v) = f_G(u,v)$ for all $u,v$.

First, let us observe that by construction, $d_{G}(u,v) \geq d_{G_{x,y}}(u,v)$ for all $u,v \in V \setminus \{x,y\}$.
The proof will proceed by establishing a series of claims, which are:
\begin{enumerate}
	\item $d_{G_{x,y}}(r,w_{x,y}) = (x,y)_{r,G}$.
	\item $d_{G_{x,y}}(r,u) = d_G(r,u)$ for all $u \in V \setminus \{x,y\}$.
	\item For all $u \in V \setminus \{x,y\}$ we have $(u,w_{x,y})_{r,G_{x,y}} = \max\{(u,x)_{r,G}, (u,y)_{r,G}\}$.
	\item $f_G(u,v) = f_{G_{x,y}}(u,v)$ for all $u,v \in V \setminus \{x,y\}$.
\end{enumerate}

To prove the first claim, let $r=u_1, u_2, \ldots, u_{k-1}, u_k = x$ be a shortest path from $r$ to $x$ in $G$.
If we consider the same path in $G_{x,y}$, only swapping $u_k$ from $x$ to $w_{x,y}$, we see that 
$$d_{G_{x,y}}(r,w_{x,y}) \leq d_G(r,x) - w_G(u_{k-1}x) + w_{G_{x,y}}(u_{k-1} w_{x,y}) \leq d_G(r,x) - (r,y)_{x,G} = (x,y)_{r,G}.$$
Let $r = v_1, v_2, \ldots, v_{k-1}, v_k = w_{x,y}$ be a shortest path from $r$ to $w_{x,y}$ in $G_{x,y}$.
By symmetry between $x$ and $y$, let us assume that $w_{G_{x,y}}(v_{k-1}w_{x,y}) = w_G(v_{k-1}x) - (r,y)_{x,G}$.
If we consider the same path in $G$, only swapping $v_k$ from $w_{x,y}$ to $x$, we see that 
\begin{eqnarray*}
(x,y)_{r,G} 	&=& 	d_G(r,x)- (r,y)_{x,G} \\
		&\geq& 	\left(d_{G_{x,y}}(r,w_{x,y}) - w_{G_{x,y}}(v_{k-1} w_{x,y}) + w_G(v_{k-1}x)\right)- (r,y)_{x,G}\\
		& =& d_{G_{x,y}}(r,w_{x,y}).
\end{eqnarray*}
This proves the first claim.

Our second claim is that $d_{G_{x,y}}(r,u) = d_G(r,u)$ for all $u \in V \setminus \{x,y\}$.
By the above observation, we need to prove that $d_G(r,u) \leq d_{G_{x,y}}(r,u)$.
By way of contradiction, assume that $d_G(r,u) < d_{G_{x,y}}(r,u)$.
Suppose that $r = v_1, v_2, \ldots, v_k = u$ is a shortest path in $G_{x,y}$.
We may assume that there exists an $i$ such that $v_i = w_{x,y}$, as otherwise this path exists in $G$ with the same weight and we are done.
By symmetry between $x$ and $y$, let us assume that $w_{G_{x,y}}(w_{x,y}v_{i+1}) = w_G(x v_{i+1}) - (r,y)_{x,G}$.
Because the algorithm deleted irrelevant edges before constructing $G_{x,y}$ we have $d_G(x, v_{i+1}) = w_{G_{x,y}}(w_{x,y}v_{i+1}) + (r,y)_{x,G}$.
Because the $v_j$ form a shortest path and $r = v_1$, we have 
$$(x,y)_{r,G} = d_{G_{x,y}}(r,w_{x,y}) = d_{G_{x,y}}(r,v_{i+1}) - w_{G_{x,y}}(w_{x,y}v_{i+1}) < d_G(r, v_{i+1}) - w_{G_{x,y}}(w_{x,y}v_{i+1}).$$
Therefore 
$$ (x,v_{i+1})_{r,G} = \frac{d_G(r,x) + d_G(v_{i+1},r) - d_G(x,v_{i+1})}2 > \frac{d_G(r,x) + (x,y)_{r,G} - (r,y)_{x,G}}2 = (x,y)_{r,G}. $$
But this contradicts the choice of $x,y$, and therefore the second claim is proven.

Our third claim is that for all $u \in V \setminus \{x,y\}$ we have $(u,w_{x,y})_{r,G_{x,y}} = \max\{(u,x)_{r,G}, (u,y)_{r,G}\}$.
By the first claim we have $d_G(x,r) = d_{G_{x,y}}(w_{x,y},r) + (r,y)_x$ and by the second claim we have $d_{G_{x,y}}(r,u) = d_G(r,u)$.
Therefore $(u,w_{x,y})_{r,G_{x,y}} - (u,x)_{r,G} = \frac{d_G(x,u) - d_{G_{x,y}}(x,u) - (r,y)_x}2$.
By construction of $G_{x,y}$ we have that $d_G(x,u) \geq d_{G_{x,y}}(w_{x,y},u) +  (r,y)_x$, and therefore $(u,w_{x,y})_{r,G_{x,y}} \geq (u,x)_{r,G}$.
By a symmetric argument we also have that $(u,w_{x,y})_{r,G_{x,y}} \geq (u,y)_{r,G}$.
Let $u = v_1, v_2, \ldots, v_{k-1}, v_k = w_{x,y}$ be a shortest path in $G_{x,y}$.
By symmetry between $x$ and $y$, suppose that $w_{G_{x,y}}(v_{k-1}w_{x,y}) = w_G(v_{k-1}x) - (r,y)_x$.
So $d_G(x,u) \leq d_{G_{x,y}}(x,u) + (r,y)_x$ and therefore $(u,w_{x,y})_{r,G_{x,y}} \leq (u,x)_{r,G}$.
This proves the third claim.

Our fourth claim is that $f_G(u,v) = f_{G_{x,y}}(u,v)$ for all $u,v \in V \setminus \{x,y\}$.
Let $u = t_1, t_2, \ldots, t_k = v$ be a path such that $f_{G}(u,v) = \min_{i}(t_i, t_{i+1})_{r,G}$.
If each instance of $x$ or $y$ in the path is replaced with $w_{x,y}$, then by the third claim we have that $f_{G}(u,v) \leq f_{G_{x,y}}(u,v)$.
Now suppose that $u = s_1, s_2, \ldots, s_k = v$ is a path such that $f_{G_{x,y}}(u,v) = \min_{i}(s_i, s_{i+1})_{r,G_{x,y}}$.
If there exists an $i$ such that $s_i = w_{x,y}$, then replace $w_{x,y}$ with vertices $s_i', s_i'' \in \{x,y\}$ such that $(s_{i-1},s_i')_{r,G} = (s_{i-1},w_{x,y})_{r,G_{x,y}}$ and $(s_{i+1},s_i'')_{r,G} = (s_{i+1},w_{x,y})_{r,G_{x,y}}$, which is possible by the third claim.
By choice of $x$ and $y$ we have that $(s_i',s_i'')_{r,G} \geq (s_{i+1},s_i'')_{r,G}, (s_{i+1},s_i'')_{r,G}$.
Therefore $f_{G}(u,v) \geq f_{G_{x,y}}(u,v)$.
This proves the fourth claim.

With the second and fourth claim, we have proven that $T_B'$ is the same tree as $T_G \setminus \{x,y\}$, with the possible exception of the placement of $w_{x,y}$ (which will then affect how $x$ and $y$ are attached to $T_B'$).
This last detail is resolved by showing that $f_{G_{x,y}}(u,w_{x,y}) = \max\{f_G(u,x), f_G(u,y)\}$, which proves that merging the path $T[r,u]$ with $T[r, w_{w_{x,y},u}]$ as performed by Gromov's algorithm on $G_{x,y}$ is the procedure as merging the path $T[r,u]$ with $T[r,w_{x,u}]$ and $T[r,w_{y,u}]$ as performed by Gromov's algorithm on $G$---and therefore $w_{x,y}$ is placed identically in both algorithms.
The proof to this detail follows a similar argument to the third and fourth claim.
\end{proof}

\section{Other Forms of Negative Curvature}\label{curve sec}

In this section, we consider several conjectures that relate different measures of graph curvature to congestion.
These conjectures were made under the context of unweighted graphs.
To simplify the presentation, if the edge weights are not described, then assume the graph is unweighted.

A version of this manuscript \cite{Y2} appeared on the internet half a decade ago, and it has underwent several changes since then.
The online version contains several original results not presented here and may be of interest.
For example, planar networks are considered in \cite{JLBB} and ``approximately planar'' in \cite{JLA} where the Euclidean grid can be held as the canonical uncongested graph, but the emphasis on the Euclidean grid as the extremal planar graph is based only on intuition.
In \cite{Y2} we confirm this intuition asymptotically by showing that for any planar graph on $n^2$ vertices the value $\max_{u \in V} \mathbb{D}(\{u\})$ is at least $\frac{8}{81\sqrt{2}}$ times the corresponding value for a square Euclidean grid of the same order.

\subsection{Symmetry and Embedding in a Surface}
Jonckheere, Lou, Bonahon, and Baryshnikov (\cite{JLBB}, Section $3$) state two conjectures, each with two parts.
We will answer all four questions.
The conjectures deal with two different concepts of graph curvature, and the conjectures relate them towards the amount of congestion in a network and the location of the highest amounts of demand in a network.
At first we will deal with the questions relating large demand to symmetry.
Afterwards we will consider the questions relating large demand to the second moment.

\begin{defn}
A graph automorphism is a permutation $f:V(G) \rightarrow V(G)$ such that $uv \in E(G)$ if and only if $f(u)f(v) \in E(G)$.
The graph is vertex-transitive if for every pair of vertices $u, v$, there exists a graph automorphism $f$ such that $f(u) = v$.
The symmetric group of a graph $G$ is the set of graph automorphism on $G$.
The symmetric group fixes a vertex $u$ if each automorphism $f$ satisfies $f(u) = u$.
\end{defn}

\begin{defn}
The \emph{inertia} of a vertex $u$ in a graph $G$ is $\sum_{w \in V} d(u,w)^2$, which is the mass for the rotational inertia if the graph was to be ``rotated'' about $u$.
This is also the second moment about $u$ (in contrast, the center of the core found in Theorems \ref{SODA full thm} and \ref{full congestion} is a vertex that minimizes the first moment).  
\end{defn}

Although we do not consider it here, the definition of inertia extends to edge weighted graphs without modification.
To the best of our knowledge, there is no ``standard'' definition of a graph automorphism in the presence of edge weights.
If the generalization to an edge weighted graph was handled in a certain way, then the proofs to the following statements would still hold.
But such an extension is perhaps a bit contrived, and unnecessary to resolve the following conjectures.

\begin{conjecture}[\cite{JLBB}, 3.4.2]
Let $G$ be a large but finite graph with positive curvature.
If the graph is vertex-transitive, then both the demand and the inertia are constant for all vertices.
\end{conjecture}

\begin{conjecture}[\cite{JLBB}, 3.3.2]\label{fixed point is not necessarily the center}
Let $G$ be a large but finite graph with negative curvature.
If the graph has a symmetric group that fixes a unique point $v_0$, then $v_0$ is the unique point of minimum inertia and maximum demand.
\end{conjecture}

We have not yet defined ``curvature'' in the sense of those conjectures, but this is not necessary.
We can prove a much stronger statement below - both by avoiding the use of curvature and generalizing the assumption about vertex-transitivity.
We say that two vertices $u,w$ are equivalent if there exists a graph automorphism $f$ such that $u = f(w)$.
The equivalency classes of this equivalence relation are known as the \emph{orbits} of the graph.
If the graph is vertex transitive, then it has one orbit, which is the entire graph---hence results about orbits is a more precise statement than a statement about vertex transitivity.

\begin{theorem} \label{graph isomorphism}
If $u,w$ are in the same orbit, then $u$ and $w$ have the same demand and inertia.
However, the smallest orbit does not correspond to the vertices with the largest demand or smallest inertia.
\end{theorem}

\begin{proof}
If $w, v_1, v_2, \ldots, v_k$ is a path in $G$, then $f(w), f(v_1), f(v_2), \ldots, f(v_k)$ is a path in $f(G)$.
Furthermore, if $T$ is a shortest-paths minimum spanning tree rooted at $w$, then $f(T)$ is a tree in $f(G)$ whose total weight is equal to $T$.
Because $G$ is finite, we also know that $f(T)$ spans $f(G)$.
Therefore the inertia for vertex $u = f(w)$ is no greater than the inertia on $w$.

On the other hand, if $v_1, \ldots v_k, w, v_1', \ldots, v_{k'}'$ is a shortest path that crosses $w$, then $f(v_1), \ldots f(v_k), f(w), f(v_1'), \ldots, f(v_{k'}')$ is a shortest path that crosses $u = f(w)$.
Therefore the demand at vertex $w$ is no more than the demand at $u = f(w)$.

If $f$ is a graph isomorphism, then so is $f^{-1}$.
Hence both of the above two inequalities are equalities.
This proves the first part of the theorem.

Now consider $H$, a finite square subgraph of the Euclidean grid.
In any graph automorphism $f$, the degree of each vertex $u$ must equal the degree of $f(u)$.
So the automorphisms of $H$ can entirely be determined by the action on the four corner vertices with degree $2$; this proves $H$ has ${4 \choose 2} = 6$ automorphisms.
If we assume that $|V(H)| = (2k)^2$, then no vertex is fixed by any of the non-trivial automorphisms.
Let $H'$ be $H$ with a vertex $u$ added, where $u$ is in $2$ edges connecting $u$ to the middle right-most vertices of $H$.

Observe that $H'$ now has only two automorphisms, and $u$ is the only vertex fixed by the non-trivial automorphism.
All other vertices are in orbits of size $2$.
It should be clear that $u$ does not minimize inertia.
Also, $\mathbb{D}(u) \leq O(k^2)$, as $u \in P_{x,y}$ only if $u \in \{x,y\}$.

On the other hand $u$ does not have the most demand among vertices in $G$, as the vertices near the center of $H$ have demand $\Omega(k^3)$.
This is proven in Theorem 8.1 of \cite{Y2}.  
A sketch of the proof is that there are $\Theta(k^2)$ vertices above and to the left of the center and an equal number below and to the right.  
So there are $\Theta(k^4)$ pairs of vertices $u,v$ where there exists a shortest path $P_{u,v}$ that crosses the center.
The proof concludes with a counting argument showing that on average a pair $u,v$ among this set of $\Theta(k^4)$ pairs contributes at least $\Omega(k^{-1})$ demand to the selected vertex in the center of the Euclidean grid.

To full refute Conjecture \ref{fixed point is not necessarily the center} we require that $H'$ have negative curvature.
For this we observe that the intuition of the above construction---begin with a graph with multiple symmetries and then append a vertex to its `outside' in a manner that it becomes the only fixed point---can be done while remaining aware of curvature.
For example, we can replace the Euclidean grid $H$ with a hyperbolic tiling of the plane ($H'$ is still formed from $H$ by adding a single vertex).
Or $H$ can be the product of a Euclidean grid and a clique.
\end{proof}

Now let us turn to the next two questions.
The first two definitions require that the graph plus a set of closed disks called \emph{faces} form a CW-complex isomorphic to an orientable $2$-manifold.
Let $|f|$ denote the number of edges incident on the boundary of a given face, $f$.
Let $d_G: V(G) \times V(G) \rightarrow \mathbb{R}^+$ be the standard graph distance with weighted edges; we will assume all edges are relevant.
As the graph is embedded in a $2$-manifold, by considering an infinitesimally small loop around a vertex $v$ we can construct a circular ordering of the edges incident with $v$.
A circular ordering of the neighbors of $v$ is an ordering of $N(v)$ that corresponds to the ordering of the edges incident with $v$ based on the embedding of the graph.
The main definition of curvature in this section comes from Alexandrov Angles.

\begin{defn} [Alexandrov Angles] \label{angles}
For a given graph $G$ and vertex $v \in V(G)$, let $\{u_1, u_2, \ldots, u_k\}$ be the neighbors of $v$ in cyclic order.
The angle $u_ivu_{i+1}$ (where $i$ is taken modulo $k$) is defined to be 
$$ \alpha_i = cos^{-1} \left( \frac{ d(v, u_i)^2 + d(v, u_{i+1})^2 - d(u_i, u_{i+1})^2 }{2d(v, u_i)d(v, u_{i+1})} \right) .$$
Under these conditions, the curvature at a vertex $v$ is
$$ k_G(v) = \frac{2 \pi - \sum \alpha_i }{\sum area(v u_i u_{i+1})} ,$$
where the area of a triangle is defined using Heron's formula:
$$ area(abc) = \frac{1}{4} \sqrt{(a+b+c)(a+b-c)(b+c-a)(a+c-b)}.$$
\end{defn}

\begin{remark} \label{equal triangles}
If the graph is unweighted and $d(u_i, u_{i+1}) = 1$ for all $i$, then 
$$k_G(v) = \frac{4 \pi}{3^{1.5}} \left( \frac{6}{d(v)} - 1 \right).$$
So if every vertex is incident to $d$ edges and satisfies the above assumptions, then it has negative curvature if $d > 6$, positive curvature if $d < 6$, and zero curvature if $d = 6$.
\end{remark}

Note that every graph can be embedded in a $2$-manifold with sufficiently large genus.
If we are discussing a graph without any particular embedding associated to it, then we will use Remark \ref{equal triangles}

Definitions of graph curvature later in this paper will be based on Definition \ref{angles}, unless stated otherwise.
A second definition of curvature comes from the genus of the space a graph embeds into, and from Euler's formula.
This second definition was proven to not measure congestion well in \cite{NS}, but we include it as a comparison.

\begin{defn} [Gaussian Curvature] \label{graph genus}
For a given graph $G$, the curvature of a vertex $v$ is
$$ k_G(v) = 1 - 0.5 d(v) + \sum_{v \in f} |f|^{-1}. $$
The curvature of the whole graph $G$ is the sum of the curvatures of the individual vertices in $G$.
This equals the Euler Characteristic of the graph (which is the sum of the number of vertices plus the number of faces minus the number of edges), which equals $2$ minus twice the genus of the $2$-manifold the graph is embedded in.
\end{defn}

It is not immediately clear what the curvature of the whole graph is in Definition \ref{angles}.
Context from Jonckheere, Lou, Bonahon, and Baryshnikov (\cite{JLBB}, Section $4$) implies that the curvature for the graph is a fixed constant, and the curvature for each vertex should equal that constant value.
This creates a bit of a conflict with their desire to examine graphs with negative curvature that are both unweighted and embedded in the plane (as opposed to $2$-manifolds with higher genus).
The trouble is that if we assume that every pair of vertices induces at most one edge, then it is well known that the average degree of the vertices is strictly less than $6$.
If every face is a triangle and every edge has weight equal to $1$, then we may apply remark \ref{equal triangles} to see that the curvature of the graph must be positive!

Clearly, one too many assumptions and simplifications have been made here.
Perhaps we should simply reduce the assumption that $G$ is planar: after all, other authors have not made this assumption previously.
In this case, if $g$ is the genus of the surface that the graph embeds into, we have that the average degree is $6 + \frac{2(g-1)}{|V(G)|}$.

A second assumption we may want to discard is that the graph is finite.
Jonckheere, Lou, Bonahon, and Baryshnikov (\cite{JLBB}, Section $4$) construct a planar graph with negative curvature by making every face a triangle and every ``internal'' vertex have degree $d \geq 7$.
The problem is that they apply their iterative construction finitely many times, which leaves them with a great plenty of non-internal vertices with degree $3$ or $4$.
Furthermore, certain parameters in their conjectures---such as the center of inertia---are not well-defined on infinite graphs.

A third assumption we may wish to discard is that the curvature of the vertices is constant.
This assumption is never explicitly mentioned.
But there does not seem to be a natural method to determine what the curvature of the graph as a whole would be in this context, aside from using the discredited Definition \ref{graph genus}.

Therefore we will always consider below that our graph is infinite or non-planar or both.
We will attempt to cover the conjectures under all of these conditions, sometimes also considering Definition \ref{graph genus} for curvature.

\begin{conjecture} [\cite{JLBB}, 3.3.1(a)] 
Let $G$ be a large but finite graph with negative curvature.
There are very few vertices with the highest demand.
\end{conjecture}

\begin{counterex}\label{transitive curve}
We argue that a graph can be vertex transitive and have negative curvature.
By Theorem \ref{graph isomorphism}, such a graph would have uniform demand.
\end{counterex}

\begin{proof}
If we drop the assumption that our graph is finite, then our counterexample is an infinite hyperbolic tiling of the two-dimensional disk.
The second option is that our graph is not planar, and our second counterexample is a lexicographic product of a cycle and a clique.
This graph has vertex set $\{u_{i,j} : 1 \leq i \leq k, 1 \leq j \leq k'\}$, with edge set $\{u_{i,j} u_{i',j'} : (i = i')\ or\ (|i-i'|=1\ and\ j=j')\}$.
It satisfies both Definition \ref{graph genus} and \ref{angles} of negative curvature.
\end{proof}

\begin{conjecture} [\cite{JLBB}, 3.3.1(b)] 
Let $G$ be a large but finite graph with negative curvature.
The vertices with the highest demand are near the vertices with the smallest inertia.
\end{conjecture}
\begin{counterex}\label{Y graph}
We argue that a graph can have negative curvature while the vertices with largest demand are arbitrarily far from those with smallest inertia.
\end{counterex}

\begin{proof}
As inertia is not well-defined for infinite graphs, we only consider the case where we drop the assumption that our graph is planar.
As the statement of the conjecture formally 
Let $G$ be a Euclidean grid over vertex set $\{(x,y): 1 \leq x \leq 4k, -h \leq y \leq h\}$, where $1 \ll h \ll k$.
Let $Y$ be $G$, but with edges $(x,0)(x,1)$ deleted for $1 \leq x \leq k$.
Observe that $Y$ represents a thick `Y,' as the vertices $\{(x,y):x>k\}$ form the long leg and the vertices $\{(x,y):x\leq k,y\leq0\}$ and $\{(x,y):x\leq k,y>0\}$ form the two short legs.
It is a direct calculation to see that the demand is maximized at the vertex $(k+1,0)$ and the inertia is minimized at $(x,0)$ for $x = (1.7+o(1))k$.
Our graph is then the lexicographic product of $Y$ and a small clique: the structure from $Y$ will imply that the vertices maximizing demand and the vertices minimizing inertia are $\Omega(k)$ apart while the edges of the clique will imply that every vertex has large degree and therefore the graph has negative curvature.
\end{proof}

\begin{conjecture}[\cite{JLBB}, 3.4.1]
Let $G$ be a large but finite graph with positive curvature.
$G$ has more balanced values for demand and inertia than a graph with negative curvature.
\end{conjecture}
\begin{counterex} \label{H graph}
Recall that we have already presented a family of graphs with negative curvature with uniform demand and inertia in Counterexample \ref{transitive curve}.
There are graphs with positive curvature with skewed distributions for demand and inertia.
\end{counterex}

\begin{proof}
We first consider an infinite graph.
While inertia and demand are not well-defined for an infinite graph, we can make the following intuitive argument about the demand.
The ringed tree (see \cite{F} for description and properties) is $5$-regular, so it has positive curvature according to Definition \ref{angles}.
Furthermore, it is planar, so it also has positive curvature according to Definition \ref{graph genus}.
But the ringed-tree has constant hyperbolicity, and therefore finite isometric subsets are congested via Theorem \ref{SODA full thm}, which is the most skewed form of demand possible.

Next we consider a finite non-planar graph.
A second example that has positive curvature by all definitions in this paper is a non-convex finite subset of the Euclidean grid with large convex subsets.
Specifically, consider a subset that represents in shape a thick `H,' similar to the thick `Y' from Counterexample \ref{Y graph}.
Each large convex subset of the graph violates any possible negative curvature property, while the small degrees force positive curvature by definitions \ref{angles} and \ref{graph genus}.
Furthermore, demand and inertia have extreme values at the small bottle necks.
\end{proof}

\subsection{Scaled Hyperbolicity}
Pestana, Rodr\'{i}gues, Sigarreta, and Villeta \cite{PRSV} have proven that a finite graph is $\delta$-hyperbolic for $\delta \leq n/4$.
An infinite graph is hyperbolic if it is $\delta$-hyperbolic for some finite $\delta$, so how do we characterize when a finite graph deserves the label ``hyperbolic?''
Answering this question is the motivation behind the definitions of scaled hyperbolicity.
As we investigate scaled hyperbolicity and possible values for various parameters, we will show that scaled hyperbolicity is inherently different than $\delta$-hyperbolicity.

For three vertices $a,b,c$, let $\vd(a,b,c) = \max\{d(a,b), d(b,c), d(a,c)\}$.
Let $P_{x,y}$ be a shortest path from $x$ to $y$, and let $I(a,b,c) = \sup_{P_{a,b}, P_{b,c}, P_{a,c}}\inf\{d(u,v) + d(v,w) + d(u,w) : u \in P_{a,b}, v \in P_{b,c}, w \in P_{a,c}\}$.
The function $I$ is related to the \emph{min-size} of the graph, which is known to be bounded above and below by the hyperbolicity (e.g., see \cite{ABCFLMSS}).
Rather than considering the min-size directly, scaled hyperbolicty considers the ratio of min-size to $\vd$.

We define the $R$-scaled hyperbolicity to be $H_R(G) = \sup_{\vd(a,b,c) > R} \frac{I(a,b,c)}{\vd(a,b,c)}$.
If $P_{a,b}$ is the shortest path of the three, then by choosing $u,v = b$ and $w = a$, we see that $H_R \leq 2$.
The primary definition of scaled hyperbolicity \cite{JLB} is $H_R$ above without the condition on taking the supremum over all shortest paths\footnote{They did not take into account that there may be many shortest paths between two vertices in a graph.  For example, between any three points $a,b,c$ in the Euclidean grid there exists shortest paths $P_{a,b}, P_{b,c}, P_{a,c}$ such that $P_{a,c} \cup P_{a,b} \cup P_{b,c}$ forms a tree.
Therefore if we do not include the condition that we want the supremum over all shortest paths between a given triple of vertices, we may call the canonical non-curved graph hyperbolic!}.
With this modified definition, a graph is then considered scaled hyperbolic if $H_R(G) < 3/2$ for an appropriately large $R$.
The constant $3/2$ is geometrically motivated, as Jonckheere, Lohsoonthorn, and Bonahon \cite{JLB} showed that Euclidean spaces and other flat networks satisfy $H_R = 3/2$, while negatively curved Riemannian manifolds satisfy $H_R < 3/2$.

In general, we require that $R>1$, because any three vertices $a,b,c$ that induce three edges satisfy $\frac{I(a,b,c)}{\vd(a,b,c)} = 2$.
Even if the graph has no triangles, a similar approach can be used with equal-distant points along the shortest cycle.
On the other hand, proofs connecting insize to hyperbolicity (see \cite{GH}) show that if $R > 16 (1 +  \delta(G))$, then $H_R(G) < 3/2$.
Thus, if $R$ is too large, then the motivation to consider scaled hyperbolicity instead of Gromov hyperbolicity dissapears.

We claim that for any fixed $R$ there exist infinitely many graphs $G$ that are $(2R/3)$-hyperbolic and $H_R(G) = 2$.
A $k$-subdivision of a graph $G$ is a graph $G^{1/k}$ such that each edge $uv$ of $G$ is replaced by a disjoint path $P_{u,v}$ that starts at $u$, ends at $v$, and has length $k$ (because we are working with unweighted graphs, this is different than the definition in Section \ref{hyperbolicity subsec}).
By the four points condition, it is easy to see that if $G$ is $\delta$-hyperbolic, then $G^{1/k}$ is $(\delta+2k)$-hyperbolic.
If $k>R/3$ and $G$ is not a tree, then every cycle in $G^{1/k}$ has length at least $R$, and so by our argument above $H_R(G^{3/R}) = 2$.
By $k$-subdividing such a $0$-hyperbolic graph with $k = \lceil R/3 \rceil$, we see that the claim is true.
Because there exists infinitely many non-tree $0$-hyperbolic graphs (they are characterized in \cite{H}), this class contains graphs that are unbounded in size.

So it must be that $R$ grows as a function of $G$.
Jonckheere, Lohsoonthorn, and Bonahon \cite{JLB} suggest that $R$ grows proportional to the diameter of the graph.
Consider the Euclidean grid on a vertex set $\{(i,j) : 0 \leq i \leq m_R, 0 \leq j \leq R/2\}$, where $m_R$ is any sequence that satisfies the desired growth rate of $R$.
Suppose $a = (a_1, a_2)$, $b = (b_1, b_2)$, and $c = (c_1, c_2)$ are three points in such a grid, and by symmetry assume that $a_1 \leq b_1 \leq c_1$.
There exists a point $(x,y) \in P_{a,c}$ such that $x = b_1$, and clearly $b \in P_{a,b} \cap P_{b,c}$.
Because $0 \leq b_2, y \leq R/2$, we see that $I(a,b,c) \leq R$ by choosing $u = b$, $v = b$, and $w = (x,y)$.
Therefore $H_R(G) \leq \frac{R}{R} < 3/2$, and the graph is scaled hyperbolic.

On the other hand, the Euclidean grid is typically held as the standard ``flat,'' or non-hyperbolic graph.
Euclidean grids do not ``enjoy such archetypical properties such as ... the confinement of quasi-geodesics in an identifiable neighborhood of the geodesic'' \cite{JLB}.
By an analysis similar to one carried out in \cite{Y2}, the maximum demand in such a graph is at most $C\frac{n^2}{m_R} \ll n^2$ for some universal constant $C$, and therefore the scaled hyperbolic family of graphs is not roughly congested.

The above examples and results also apply to alternate versions of scaled hyperbolicity, such as the scaled four points condition of Jonckheere, Lohsoonthorn, and Ariaei \cite{JLA}.

\end{document}